\documentclass[a4paper, 11pt]{article}

\usepackage{amsmath}
\usepackage{amsthm}
\usepackage{amsfonts}
\usepackage{amssymb}
\usepackage{latexsym}
\usepackage{color}
\usepackage{mathrsfs}
\usepackage{enumerate}
\usepackage{ulem, bm}
\usepackage{hyperref}
\hypersetup{colorlinks=true,urlcolor=black,linkcolor=black,citecolor=black}

\usepackage{mathtools}
\mathtoolsset{showonlyrefs=true}

\allowdisplaybreaks

\allowdisplaybreaks[4]
\newcommand{\form}{{\mathcal E}}
\newcommand{\dom}{{\mathcal F}}

\usepackage{version}

\numberwithin{equation}{section}

\newtheorem{theorem}{Theorem}[section]
\newtheorem{proposition}[theorem]{Proposition}
\newtheorem{lemma}[theorem]{Lemma}
\newtheorem{corollary}[theorem]{Corollary}

\theoremstyle{definition}
\newtheorem{definition}[theorem]{Definition}
\newtheorem{example}[theorem]{Example}
\newtheorem{remark}[theorem]{Remark}

\newcommand{\vareps}{\varepsilon}
\newcommand{\dis}{\displaystyle}

\newcommand{\ds}{\displaystyle}
\newcommand{\Capa}{{\rm Cap\,}}

\newcommand{\nest}{\mbox{\small {\sf N} \hspace*{-12.5pt} {\sf N}}}

\renewcommand{\c}[1]{\mathcal{#1}}
\renewcommand{\b}[1]{\mathbb{#1}}

\title{
  Classification and Metrization of Classes of Smooth measures}

\author{\sf
Takumu Ooi\thanks{Department of Mathematics,
Faculty of Science and Technology, Tokyo University of Science,
Noda, Chiba, 278-8510, Japan
({\sf ooitaku@rs.tus.ac.jp})},  \  Kaneharu Tsuchida\thanks{Department of Mathematics,
National Defense Academy,
Yokosuka, Kanagawa, 239-8686, Japan
({\sf tsuchida@nda.ac.jp}) } \ 
{\rm and}  \ Toshihiro Uemura\thanks{Department of Mathematics, 
Faculty of Engineering Science, 
Kansai University, Suita, Osaka, 564-8680, Japan
({\sf t-uemura@kansai-u.ac.jp})
}}

\date{}
\begin{document}

\maketitle 
\baselineskip=16pt 

\vspace*{-1cm}

\begin{abstract}
We classify the several classes of the set of smooth measures from the perspective of the denseness and the locality, and consider their relationships, in particular, that of the Kato class and Radon measures of finite energy integrals. We also introduce the Miyadera metric on the Dynkin class, and obtain the continuity of the Revuz correspondence.

\end{abstract}

\section{Introduction}
The Kato class was originally introduced
by Tosio Kato 
for the study of Schr\"{o}dinger operators. 
In particular, it played a pivotal role 
in establishing the essential self-adjointness 
of the Schr\"{o}dinger operator $H = -\Delta + V$, 
even in the presence of highly singular potentials. 
While the Kato class was initially defined 
through analytical integrability conditions, 
the influential contribution of Aizenman and Simon \cite{AS82} 
provided a fundamental probabilistic characterization. 
From a probabilistic perspective, 
the Kato class provides a natural setting 
for ensuring the stability of the Feynman-Kac functional. 
Specifically, it guarantees that 
the associated Schr\"{o}dinger semigroup $e^{-tH}$ 
is well-defined and possesses 
a bounded, continuous integral kernel, 
which allows for a rigorous representation 
of the operator through expectations 
over Brownian paths.

Since then, the study of the Kato class 
and smooth measures has undergone 
significant development within the theory 
of Dirichlet forms and symmetric Markov processes. 
In the framework of symmetric Markov processes, 
smooth measures play a central role 
due to their bijective correspondence 
with positive continuous additive functionals (PCAFs) 
via the Revuz correspondence. 
This link is vital for various 
transformations of Markov processes. 
In this context, Albeverio and Ma \cite{AM92} 
studied the structure of smooth measures 
associated with general Dirichlet forms, 
revealing that any smooth measure 
can be approximated by Kato class measures. 
Stollmann and Voigt \cite{SV96} introduced 
an extended Kato class and deeply analyzed 
measure perturbations of Dirichlet forms 
using an operator-theoretical approach. 
Furthermore, Kuwae and Takahashi \cite{KT07} 
completely extended the results 
of Aizenman and Simon to general 
symmetric Markov processes, 
proving that different definitions 
of the Kato class coincide 
under the mild condition of having 
upper and lower heat kernel estimates.

Within this family of smooth measures, 
there are two specific subclasses 
frequently utilized due to their 
desirable analytical properties. 
One is the class $\mathcal{S}_0$, 
which consists of measures 
with finite energy integrals, 
and the other is its subclass $\mathcal{S}_{00}$, 
characterized by measures 
with finite total mass 
and bounded $1$-potentials. 
Although these classes have been 
extensively studied in potential theory, 
their precise relationship with the Kato class 
has remained somewhat elusive until now.

The purpose of this paper is 
to clarify these connections 
within the framework of symmetric Markov processes 
and to provide a comprehensive comparison 
of several classes of smooth measures such as the Kato class $\mathcal{K}$, the Dynkin class $\mathcal{D}$, \(\mathcal{S}_0\) and \(\mathcal{S}_{00}\). Specifically, we first consider the denseness of several classes in the set of smooth measures and the locality of classes, and we clarify them by treating the denseness and the locality as if they were in a dual relationship. Moreover we investigate the analytical properties of these measures 
and establish a sufficient condition 
for a general smooth measure to be a Radon measure 
based on the boundedness of its potentials. 
As one of our main results, we rigorously prove that 
any measure in the Kato class with finite total mass 
(denoted as $\mathcal{K}_{00}$) necessarily belongs to $\mathcal{S}_{00}$. 
This finding clarifies the inclusion relation 
between these two fundamental classes ($\mathcal{K}_{00} \subset \mathcal{S}_{00}$) 
and provides a solid bridge 
between Kato-type integrability 
and the finiteness of energy integrals. 
Furthermore, we demonstrate that the converse inclusion 
($\mathcal{S}_{00} \subset \mathcal{K}_{00}$) does not hold in general 
by providing an explicit counterexample 
associated with a pure jump step process. 
We also show, however, that 
under additional regularity conditions such as the Feller property 
of the associated semigroup,
measures in $\mathcal{S}_{00}$ do indeed belong to the Kato class.
These analytical results are then extended to the case of 
transient Dirichlet forms.
To illustrate the delicate balance 
between singularity and integrability, 
we present a detailed analysis of the singular measure 
$\mu(dx) = |x|^{-\beta} dx$ 
for the $d$-dimensional Brownian motion, 
explicitly determining the exact range of the exponent $\beta$ 
for which it belongs to the class $\mathcal{S}_0$ or $\mathcal{K}$.
Interestingly, while the intersection of the Kato class and
$\mathcal{S}_0$ is non-empty in dimension $d=3$,
these two classes become completely disjoint in higher dimensions $d \ge 4$.
The above example illustrates that, in contrast to the finite measure case,
an infinite measure in the Kato class does not readily belong to $\mathcal{S}_0$.

In the study of singular measures such as those in the Kato class $\mathcal{K}$
and $\mathcal{S}_{0}$, the Stollmann-Voigt inequality serves as 
a fundamentally powerful tool. Recognized as a transient Poincaré-type inequality 
\cite[p.459, Notes to Chapter 2]{FOT11}, this estimate bounds the $L^2$-norm 
with respect to a measure $\mu$ by the Dirichlet energy 
(i.e., $\int_E f^2 d\mu \le \|R_1\mu\|_\infty \mathcal{E}_1(f,f)$). 
It plays a crucial role in analyzing measure perturbations of Dirichlet forms, 
particularly in establishing relative boundedness and investigating the bottom 
of the spectrum (and the spectral gap) of the associated Schrödinger operators. 
This inequality was originally established by Stollmann and Voigt \cite{SV96}
with constant $4$ through an operator-theoretical approach.
Later, Ben Amor \cite{BA04} 
proved the inequality with the sharp constant $1$ 
without assuming the absolute continuity condition ((AC) condition) 
of the transition semigroup, by heavily relying on 
potential-theoretical tools for regular Dirichlet forms.
In this paper, we provide an alternative and self-contained proof 
of this sharp inequality. Instead of reducing the problem 
to the boundedness or compactness of operators in $L^2$-spaces, 
our approach directly utilizes the duality expression of Dirichlet forms 
combined with a pointwise Cauchy-Schwarz estimate of potentials 
via their probabilistic representations. 
This method allows us to avoid the (AC) condition 
and directly obtain the sharp constant $1$, 
highlighting the elegant interplay between the analytic duality 
and the probabilistic structure of symmetric Markov processes.

In addition to these analytical comparisons, 
we explore the topological structures 
of the Dynkin class, the Kato class and their related classes of measures. 
We first demonstrate that a metric based on the supremum norm 
of the $1$-potentials is not complete. 
Instead, by introducing a specific metric $d_{\mathcal{D}}$ 
on the Dynkin class based on the $1$-potential 
of the total variation measure, 
we successfully prove that the Dynkin class 
forms a complete metric space. This metric $d_{\mathcal{D}}$ is a generalization of the Miyadera norm \cite{OSSV96} to the space of measures.
Under this metric, we establish that the Kato class 
and the class of Green-tight measures 
are closed subsets of the Dynkin class. We also prove the continuity of the Revuz correspondence under this metric \(d_\mathcal{D}\).
Furthermore, we introduce another metric 
for classes of smooth measures attached to compact nests 
and demonstrate their completeness. 
Through these analytical and topological investigations, 
this paper aims to provide a modernized 
and profound understanding of the Kato class 
and smooth measures.

The organization of this paper is as follows. Section \ref{sec_col_loc} is devoted to the classification of several classes of the set \(\mathcal{S}\) of all smooth measures. It is well-known that for any \(\mu \in \mathcal{S}\), we can take a nest \(\{F_k\}_k\) such that \(1_{F_k} \mu \in \mathcal{S}_0\). By considering the localization, we show that the same property holds for various classes. In Section \ref{sec_Radon_Kato}, we consider the relationship between classes of smooth radon measures, in particular, the set of all Radon measures of finite energy integrals and the Kato class. In Section \ref{sec_Miyadera}, we consider topologies on the Dynkin class and we introduce the Miyadera metric. We show that the completeness of the Dynkin class with the Miyadera metric. We also show the continuity of the Revuz correspondence under this metric. Appendix \ref{sec_Appndeix} is devoted to the preliminaries of Dirichlet form theory.

\section{Classification of classes of Smooth measures}\label{sec_col_loc}
Throughout this section, we fix a regular Dirichlet form \((\mathcal{E}, \mathcal{F})\) on \(L^2(E;m)\). We call an increasing sequence of closed (resp. compact) sets $\{F_k\}$ a {\it nest} (resp. {\it compact nest}) if \(\Capa (K\setminus F_k)\) converges to \(0\) as \(k\) tend to infinity for any compact set \(K\), where \(\Capa\) is the capacity associated with \((\mathcal{E}, \mathcal{F})\). We remark that this nest is called a generalized nest in \cite{FOT11}. Denote by \(\nest\) the family of all compact nests. See Appendix \ref{sec_Appndeix} for definitions in Dirichlet form theory.

A positive Borel measure $\mu$ on $E$ is {\it smooth} if it does not charge on a set of zero capacity and there exists a nest \(\{F_k\}_ k \in \nest\) satisfying \(\mu(F_k) <\infty\) for each \(k\). Denote by $\c{S}$ the class of all smooth measures. There are many classes of \(\mathcal{S}\), and it is known that some classes \(\mathcal{T}\)  are dense in \(\mathcal{S}\) in the following sense: For any smooth measure \(\mu\), there exists a nest \(\{F_k\}_k\) such that \(1_{F_k}\mu \in \mathcal{T}\) for each \(k\).

In this section, we classify the classes of \(\mathcal{S}\) from the perspective of the denseness and the localization, the latter being, in a sense, counterpart of the former.

For \(\mathcal{T} \subset \mathcal{S}\) and \(\{F_k\}_k \in \nest\), denote by \(\mathcal{T}(\{F_k\})\) the set of all positive Borel measures \(\mu\) charging no set of zero capacity satisfying \(1_{F_k} \mu \in \mathcal{T}\) for each \(k\). Moreover, we define its collection \(\mathcal{T}_{col}\) taken over all compact nests by
\[\mathcal{T}_{col}:= \bigcup_{\{F_k\} \in \nest} \mathcal{T}(\{F_k\}).\]

We remark that \(\mathcal{T}(\{F_k\}) \subset \mathcal{S}\) and \(\mathcal{T}_{col} \subset \mathcal{S}\) hold for any \(\mathcal{T} \subset \mathcal{S}\). Indeed, for \(\mu \in \mathcal{T}(\{F_k\})\), since \(1_{F_k}\mu\) is smooth for each \(k\), there exists \(G_\ell^{(k)} \in  \nest\) such that \(\mu(F_k \cap G_\ell^{(k)}) < \infty\) for each \(\ell\). For an increasing sequence of relatively compact open sets \(\{U_n\}_n\) satisfying \(E = \bigcup_n U_n\), we take a sequence \(\{\ell(k)\}\) satisfying \(\Capa(\overline{U}_n \setminus G_{\ell(k)}^{(k)}) \leq 2^{-k}\) for any \(n \leq k\). For any compact set \(K\), we take \(n\) satisfying \(K\subset U_n\) and an integer \(k \geq n\). Then \(\Capa(K \setminus \bigcup_{i=1}^k (F_i \cap G_{\ell(i)}^{(i)}))\leq \Capa(K \setminus (F_k \cap G_{\ell(k)}^{(k)})) \leq \Capa(\overline{U}_n \setminus G_{\ell(k)}^{(k)}) + \Capa(K \setminus F_k) \leq 2^{-k} + \Capa(K \setminus F_k) \) converges to \(0\) as \(k\) tends to infinity. Hence \(\{\bigcup_{i=1}^k (F_i \cap G_{\ell(i)}^{(i)})\}_k\) is a compact nest satisfying \(\mu(\bigcup_{i=1}^k (F_i \cap G_{\ell(i)}^{(i)}))<\infty\), and so \(\mu\) is smooth. In particular, \(\mathcal{S}(\{F_k\})= \mathcal{S}\) and \(\mathcal{S}_{col} = \mathcal{S}\) hold for any \(\{F_k\}_k \in \nest\).

We also define the local class \(\mathcal{T}_{loc}\) of \(\mathcal{T}\) by
\[\mathcal{T}_{loc}:=\{\mu \in \mathcal{S} : 1_K\mu\in \mathcal{T} \text{\ for\ any\ compact\ set\ }K\}.\]

\begin{proposition} For any \(\mathcal{T}\subset \mathcal{S}\), it holds that
\[\mathcal{T}_{loc} = \bigcap_{\{F_k\} \in \nest} \mathcal{T}(\{F_k\}).\]
\end{proposition}

\begin{proof}
For any compact nest \(\{F_k\}\) and any \(\mu \in\mathcal{T}_{loc}\), since \(1_{F_k} \mu \in \mathcal{T}\), it holds that \( \mathcal{T}_{loc}  \subset \mathcal{T}(\{F_k\})\) and so \(\mathcal{T}_{loc} \subset \bigcap_{\{F_k\} \in \nest} \mathcal{T}(\{F_k\})\).

We take \(\mu \in \bigcap_{\{F_k\} \in \nest} \mathcal{T}(\{F_k\})\) and any compact set \(K\). For a compact nest \(\{F_k\}\), we set \(K_1:=K\) and \(K_k:=K\cup F_k\) for \(k\geq 2\), then \(\{K_k\}_k\) is also a compact nest. Hence \(1_{K}\mu = 1_{K_1}\mu \in \mathcal{T}\) and so \(\mathcal{T}_{loc} \supset \bigcap_{\{F_k\} \in \nest} \mathcal{T}(\{F_k\})\).
\end{proof}

\begin{remark}
For any compact nests \(\{F_k\}_k\) and \(\{G_k\}\), a sequence \(\{F_k \cap G_k\}_k\) is also a compact nest since \(\Capa(K\setminus (F_k \cap G_k)) \leq \Capa(K\setminus F_k) +\Capa(K\setminus G_k)\) for any compact set \(K\). Then the set of all compact nest \(\nest\) is a directed set by defining the order \(\{F_k\}_k \leq \{G_l\}_l \) by, for any \(l\), there exists \(k\) such that \(F_k \supset G_l\).

We assume that, for any \(\mu \in \mathcal{T}\) and any compact set \(K\), \(1_K\mu \in \mathcal{T}\). Then, for any compact nests \(\{F_k\}_k\) and \(\{G_l\}_l\) with \(\{F_k\}_k \leq \{G_l\}_l\), it holds that \(\mathcal{T}(\{F_k\}) \subset \mathcal{T}(\{G_l\})\). Hence \(\{\mathcal{T}(\{F_k\})\}_{\{F_k\}_k \in \nest}\) is a direct system with identity maps, and so we can define an inductive limit 
\[\varinjlim \mathcal{T}(\{F_k\}) := \bigsqcup_{\{F_k\}\in \nest} \mathcal{T}(\{F_k\})/ \sim\]
where the equivalence \(\mu \sim \nu\) for \(\mu \in \mathcal{T}(\{F_k\}) \) and \(\nu \in \mathcal{T}(\{G_l\})\) is defined by the existence of some \(\{H_n\}_n \in \nest\) such that \(\mu,\nu \in \mathcal{T}(\{H_n\})\) and \(1_{H_n}\mu = 1_{H_n}\nu\) for any \(n\). This inductive limit coincides with \(\mathcal{T}_{col}\), and the inductive limit is also called the colimit. The subscript ``col" of \(\mathcal{T}_{col}\) stands for both ``colection" and ``colimit".

 We emphasize that set theoretic or category theoretic discussions, such as inductive limits or colimits, will not be used in this paper except for this remark. The terminology of inductive limits and related notions is used only to clarify the reason why we use the notion \(\mathcal{T}_{col}\).
\end{remark}

The following properties hold.
\begin{proposition}\label{Prop_loccol_1}
For any subclasses \(\mathcal{T}, \mathcal{T}_n\) and \(n\in \mathbb{N}\), the following hold.
\begin{enumerate}
\item \(\left(\bigcap_n \mathcal{T}_n \right)_{loc} = \bigcap_n (\mathcal{T}_n)_{loc}\) and \(\left(\bigcup_n \mathcal{T}_n \right)_{loc} \supset \bigcup_n (\mathcal{T}_n)_{loc}\).
\item \(\left(\bigcap_n \mathcal{T}_n \right)_{col} \subset \bigcap_n (\mathcal{T}_n)_{col}\) and \(\left(\bigcup_n \mathcal{T}_n \right)_{col} = \bigcup_n (\mathcal{T}_n)_{col}\).
\item If \(\mathcal{T}_1 \subset \mathcal{T}_2\), then \((\mathcal{T}_{1})_{loc} \subset (\mathcal{T}_{2})_{loc}\) and  \((\mathcal{T}_{1})_{col} \subset (\mathcal{T}_{2})_{col}\).
\item For any \(\{F_k\}_k \in \nest\), \(\mathcal{T}_{loc} \subset \mathcal{T}(\{F_k\}) \subset \mathcal{T}_{col}\).
\item \((\mathcal{T}_{loc})_{loc} = \mathcal{T}_{loc}\) and \((\mathcal{T}_{col})_{col} = \mathcal{T}_{col}\).
\item \(\mathcal{T}_{loc} \subset (\mathcal{T}_{loc})_{col} \subset \mathcal{T}_{col}\) and \(\mathcal{T}_{loc} \subset (\mathcal{T}_{col})_{loc} \subset \mathcal{T}_{col}\).
\end{enumerate}
\end{proposition}

\begin{proof}
(1), (2), (3) and (4) are clear.

We prove (5). For any \(\mu \in \mathcal{T}_{loc}\) and a compact set \(K\), \(1_K \mu\) belongs to \(\mathcal{T}\). For any compact set \(L\), since \(K\cap L\) is also a compact set,  \(1_L(1_K\mu)=1_{K\cap L}\mu\) belongs to \(\mathcal{T}\) and so \(\mu \in (\mathcal{T}_{loc})_{loc}\). For any \(\mu \in (\mathcal{T}_{loc})_{loc}\) and compact set \(K\), \(1_K\mu \in \mathcal{T}_{loc}\). Hence we have \(1_K \mu = 1_K(1_K\mu) \in \mathcal{T}\) and so \(\mu \in \mathcal{T}_{loc}\).

 For any \(\mu \in \mathcal{T}_{col}\), there exists \(\{F_k\}_k\in \nest\) such that \(1_{F_k} \mu \) belongs to \(\mathcal{T}\) for each \(k\). Then we have \(1_{F_k}(1_{F_l}\mu)=1_{F_{k\wedge l}}\mu\) belongs to \(\mathcal{T}\) for each \(l\), and so \(\mu \in (\mathcal{T}_{col})_{col}\). For any \(\mu \in (\mathcal{T}_{col})_{col}\), there exists \(\{F_k\}_k\in \nest\) such that \(1_{F_k} \mu \) belongs to \(\mathcal{T}_{col}\) for each \(k\), and there exists \(\{G_l^{(k)}\}_l \in \nest\) such that \(1_{G_l^{(k)}} 1_{F_k} \mu \in \mathcal{T}\). We set \(\{A_j\}_j\) as \(\{G_l^{(k)} \cap F_k\}_{k,l}\) and \(B_n:=\cup_{j=1}^n A_j\). Then, for any compact set \(K\) and \(\varepsilon>0\), we take \(k_0, l_0\) satisfying \(\Capa(K\setminus F_{k_0})\leq \varepsilon\) and  \(\Capa((K \cap F_{k_0})\setminus G^{(k_0)}_{l_0})\leq \varepsilon\). For large \(n\) satisfying \(G^{(k_0)}_{l_0} \cap F_{k_0} \subset B_n,\)
 we have \(\Capa(K\setminus B_n) \leq \Capa((K \cap F_{k_0}) + \Capa(K \cap F_{k_0})\setminus G^{(k_0)}_{l_0}) \leq \varepsilon\). Hence \(\{B_n\}_n\) is a compact nest and \(1_{B_n}\mu \in \mathcal{T}\) and so \(\mu \in \mathcal{T}_{col}\).
 
(6) follows from (3) and (5).
\end{proof}

\begin{proposition}\label{Prop_loccol_3}
We assume that \(1_K \mu \in \mathcal{T}\) for any \(\mu \in \mathcal{T}\) and compact set \(K\). Then the following holds.
\begin{enumerate}
\item \(\mathcal{T}\subset \mathcal{T}_{loc} \subset \mathcal{T}_{col}\).
\item \((\mathcal{T}_{loc})_{col} = (\mathcal{T}_{col})_{loc}= \mathcal{T}_{col}\).
\end{enumerate}
\end{proposition}
\begin{proof}
(1) is clear. By (1), Proposition \ref{Prop_loccol_1} (3) and (5), we have \(\mathcal{T}_{col} \subset (\mathcal{T}_{loc})_{col} \subset \mathcal{T}_{col}\). By setting \(\mathcal{T}_{col}\) in place of \(\mathcal{T}\) in (1), we have  \(\mathcal{T}_{col} \subset (\mathcal{T}_{col})_{loc} \subset \mathcal{T}_{col}\), and so (2) holds.
\end{proof}

\begin{corollary}
We assume that \(1_K \mu \in \mathcal{T}\) for any \(\mu \in \mathcal{T}\) and compact set \(K\).  If \(\mathcal{T}=\mathcal{T}_{col}\) then \(\mathcal{T}=\mathcal{T}_{loc}\).
\end{corollary}
\begin{proof}
By Proposition \ref{Prop_loccol_3}, \(\mathcal{T}=\mathcal{T}_{col}\) implies \(\mathcal{T}_{loc}=(\mathcal{T}_{col})_{loc}=\mathcal{T}_{col}\).
\end{proof}

For a smooth measure $\mu\in {\cal S}$ and its corresponding PCAF \(A^\mu\), we define
\[ R^{\mu}_{\alpha}f(x) := \mathbb{E}_x\left[\int_0^{\infty} e^{-\alpha t}f(X_t)\,dA_t^{\mu} \right]\]
for a measurable function \(f\) and \(\alpha >0\). In particular, we set \(R_{\alpha}\mu = R_{\alpha}^\mu 1\). See Appendix \ref{sec_Appndeix} for details for PCAFs.

For a measurable function \(u\),  $\|u\|_{\infty, q}$ denotes the {\it quasi-essential supremum} norm of a function $u$ defined by
$$
\|u\|_{\infty,q}:= \inf\Big\{ \lambda>0 : \ {\rm Cap}\big( \{x\in E : \ |u(x)| \ge \lambda \} \big)=0 \Big\}.
$$

We use the following definition of Kato class measures given in \cite{AM92}.

\begin{definition}
\begin{enumerate}
\item A smooth measure \(\mu\) is {\it a Kato class measure} if it satisfies \( \lim_{\alpha \to \infty}\|U_\alpha \mu\|_{\infty, q} =0.\) Denote by \(\mathcal{K}\) the set of all Kato class measures.
\item  A smooth measure \(\mu\) is {\it a Dynkin class measure} if it satisfies \(\|U_1 \mu\|_{\infty, q} < \infty \). Denote by \(\mathcal{D}\) the set of all Dynkin class measures.
\item A Kato class measure \(\mu\) is {\it a Green-tight measure} for \(R_1\) if, for any \(\varepsilon >0\), there exists a compact set \(K\) such that \(\|R_1(1_{K^c}\mu)\|_\infty \leq \varepsilon\). Denote by \(\mathcal{K}_\infty= \mathcal{K}_\infty(R_1)\) the set of all Green-tight measures.

\item A Radon measure \(\mu\) has {\it finite energy integrals} if there exists \(C>0\) such that, for any \(f \in \mathcal{F}\cap C_c\), \(\int |f|d\mu \leq C\sqrt{\mathcal{E}_1(f,f)}\). Denote by \(\mathcal{S}_0\) the set of all Radon measures of finite energy integrals.

\item When $(\form,\dom)$ is transient, a Radon measure \(\mu\) has {\it \(0\)-order finite energy integrals} if there exists \(C>0\) such that, for any \(f \in \mathcal{F}\cap C_c\), \(\int |f|d\mu \leq C\sqrt{\mathcal{E}(f,f)}\). Denote by \(\mathcal{S}_0^{(0)}\) the set of all Radon measures of finite \(0\)-order energy integrals.
\end{enumerate}
\end{definition}

For \(\mu \in \mathcal{S}_0\) (resp. \(\mathcal{S}_0^{(0)}\)) and \(\alpha > 0\) (resp. \(\alpha=0\)), by the Riesz representation theorem, there exists an \(\alpha\)-order potential \(U_\alpha \mu \in \mathcal{F}\) such that \(\mathcal{E}_\alpha(f,U_\alpha \mu) = \int \tilde{f} d\mu\) for any \(f\in \mathcal{F}\) (resp. \(f\in \mathcal{F}_e\)) and its quasi-continuous version \(\tilde{f}\). Denote by \(\mathcal{S}_{00}\) the set of all \(\mu \in \mathcal{S}_0\) satisfying \(\mu(E)< \infty\) and \(U_1\mu \in L^\infty.\) Moreover we define \(\mathcal{K}_0:=\mathcal{K} \cap \mathcal{S}_0\), \(\mathcal{K}_{00}:=\mathcal{K} \cap \mathcal{S}_{00}\),  \(\mathcal{D}_0:=\mathcal{D} \cap \mathcal{S}_0\), \(\mathcal{D}_{00}:=\mathcal{D} \cap \mathcal{S}_{00}\), and denote by  \(\mathcal{S}_R\) the set of all smooth Radon measures and \(\mathcal{S}_F\) the set of all finite smooth measures,

We remark that, by the resolvent equation \(\|U_\alpha \mu\|_{\infty,q}<\infty\) for any \(\mu \in \mathcal{D}\) and \(\alpha >0\), and \(\mathcal{K} \subset \mathcal{D}\) holds.

\begin{remark} The following holds for the quasi-essential supremum norm: for a function $v$ defined on $E$ quasi-everywhere, 
$$
\| v\|_{\infty,q}=\inf_{{\rm Cap}(N)=0} \sup_{x\in E\setminus N} |v(x)|.
$$
In fact, for any $\lambda >\|v\|_{\infty,q}$, the set $N_\lambda:=\{x\in E: |v(x)|>\lambda\}$ satisfies ${\rm Cap}(N_\lambda)=0$. Taking $\lambda_n \downarrow \|v\|_{\infty, q}$ and setting $N=\bigcup_{n=1}^\infty N_{\lambda_n}$, we have ${\rm Cap}(N)=0$ and $\sup_{x\in E\setminus N} |v(x)| \le \|v\|_{\infty, q}$.

Conversely, for any $N\subset E$ with ${\rm Cap}(N)=0$, let $\lambda = \sup_{E \setminus N} |v(x)|$. Then $\{x \in E : |v(x)| > \lambda\} \subset N$, hence its capacity is zero. By the definition of $\|v\|_{\infty,q}$, we have $\|v\|_{\infty,q} \le \lambda$, which leads to $\|v\|_{\infty,q} \le \inf_{{\rm Cap}(N)=0} \sup_{x\in E \setminus N} |v(x)|$.
\end{remark}

\medskip
We introduce a simple but useful lemma, which appeared as Lemma 4.1 in \cite{AM92}. For the reader's convenience, we provide a proof here.

\begin{lemma}[{\cite[Lemma 4.1]{AM92}}] \label{potCapalemma}
For a smooth measure \(\mu \in \mathcal{S}\) and its corresponding PCAF \(A^\mu\), the following equalities hold:
\[
\|R_\alpha \mu\|_{\infty, q} = \|R_\alpha \mu\|_{\infty} \quad \text{for any } \alpha > 0,
\]
and
\[
\|\mathbb{E}_{\cdot}[A_t^{\mu}]\|_{\infty, q} = \|\mathbb{E}_{\cdot}[A_t^{\mu}]\|_{\infty} \quad \text{for any } t > 0.
\]
Here, $\|\cdot\|_{\infty}$ denotes the essential supremum norm with respect to the underlying measure $m$.
\end{lemma}

\begin{proof}
Since any Borel set $N$ with $\mathrm{Cap}(N)=0$ satisfies $m(N)=0$, the inequality $\|R_\alpha \mu\|_{\infty, q} \geq \|R_\alpha \mu\|_{\infty}$ follows immediately. 

To prove the reverse inequality, we may assume that \(C:=\|R_\alpha \mu\|_{\infty} < \infty\). By \cite[Theorem 2.2.4]{FOT11}, there exists an increasing sequence of compact sets  \(\{F_n\}_n\) such that \(\mu_n:=1_{F_n}\mu \in \mathcal{S}_0, \mu(E\setminus \bigcup _{n=1}^\infty F_n)=0\) and \(\lim_{n\to \infty}\Capa(K\setminus F_n)=0\) for any compact set $K\subset E$.  For each $n$, $R_\alpha \mu_n$ is quasi-continuous and satisfies $R_\alpha \mu_n \le R_\alpha \mu \le C$ $m$-a.e., which implies $R_\alpha \mu_n(x) \le C$ for quasi-every $x \in E$ from \cite[Lemma 2.1.5.]{FOT11}.  
Note here that $A_t^\mu$ does not charge $E \setminus \bigcup F_n$  q.e. because $\mu(E\setminus \bigcup _{n=1}^\infty F_n)=0$. 
Then by the monotone convergence theorem, we have $R_\alpha \mu_n(x) \uparrow R_\alpha \mu(x)$ for quasi-every $x \in E$.  
Combining these, we obtain $R_\alpha \mu(x) \le C$ for quasi-every $x \in E$, which implies $\|R_\alpha \mu\|_{\infty, q} \le C$.

Similarly, we can prove \(\|\mathbb{E}_{\cdot} [A_t^{\mu}]\|_{\infty, q} = \|\mathbb{E}_{\cdot} [A_t^{\mu}]\|_{\infty}\).
\end{proof}

By virtue of this lemma, we can use the essential supremum norm $\|\cdot\|_{\infty}$ and the quasi-essential supremum norm $\|\cdot\|_{\infty,q}$ interchangeably when considering the potentials of Kato class measures.

\begin{lemma}\label{lemdefKato}It holds that
\begin{eqnarray*}
\mathcal{K} \!\! &=& \big\{\mu \in \mathcal{S} \mid \lim_{t \to 0}\|\mathbb{E}_{\cdot}[A_t^\mu]\|_{\infty} =0 \big\}\  =\  \big\{\mu \in \mathcal{S} \mid \lim_{t \to 0}\|\mathbb{E}_{\cdot}[A_t^\mu]\|_{\infty,q } =0 \big\}. 
\end{eqnarray*}
\end{lemma}

\begin{proof}
By Lemma \ref{potCapalemma}, it is enough to show that the first equation. 
First, assume \(\mu \in \mathcal{K}\). For any \(\alpha > 0\) and \(t > 0\), we have the elementary inequality:
$$
\mathbb{E}_x[A_t^\mu] = \mathbb{E}_x\left[ \int_0^t dA_s^\mu \right] \leq e^{\alpha t} \mathbb{E}_x\left[ \int_0^t e^{-\alpha s} dA_s^\mu \right] \leq e^{\alpha t} R_\alpha \mu(x).
$$
Taking the essential supremum norm, we obtain $\|\mathbb{E}_{\cdot}[A_t^\mu]\|_{\infty} \leq e^{\alpha t} \|R_\alpha \mu\|_{\infty}$. 
Since \(\mu \in \mathcal{K}\), taking the limit superior as \(t \downarrow 0\) yields \(\limsup_{t \downarrow 0} \|\mathbb{E}_{\cdot}[A_t^\mu]\|_{\infty} \leq \|R_\alpha \mu\|_{\infty}\) for each fixed $\alpha > 0$.  By letting \(\alpha \to \infty\), we conclude \(\lim_{t \to 0} \|\mathbb{E}_{\cdot}[A_t^\mu]\|_{\infty} = 0\).

Conversely, assume \(\lim_{t \to 0} \|\mathbb{E}_{\cdot}[A_t^\mu]\|_{\infty} = 0\) for a smooth measure $\mu$. As in the proof of Lemma \ref{potCapalemma}, take a nest \(\{F_n\}\) such that \(\mu_n := 1_{F_n}\mu \in \mathcal{S}_0\) and \(\mu(E \setminus \bigcup F_n) = 0\). For each \(n\), the potential \(R_\alpha \mu_n\) is quasi-continuous. Using the Markov property, for any \(t > 0\), we have:
\begin{eqnarray*}
R_\alpha \mu_n(x) &=& \mathbb{E}_x\left[\int_0^t e^{-\alpha s} dA_s^{\mu_n} \right] + \mathbb{E}_x\left[ e^{-\alpha t} R_\alpha \mu_n(X_t) \right] \\
&\leq& \mathbb{E}_x[A_t^{\mu_n}] + e^{-\alpha t} \|R_\alpha \mu_n\|_{\infty, q}.
\end{eqnarray*}
Taking the quasi-essential supremum over \(x\) and using Lemma \ref{potCapalemma},
we get \(\|R_\alpha \mu_n\|_{\infty, q} \leq \|\mathbb{E}_{\cdot}[A_t^{\mu}]\|_{\infty} + e^{-\alpha t} \|R_\alpha \mu_n\|_{\infty, q}\). Rearranging this gives:
\[ \|R_\alpha \mu_n\|_{\infty, q} \leq \frac{1}{1 - e^{-\alpha t}} \|\mathbb{E}_{\cdot}[A_t^\mu]\|_{\infty}. \]
Since \(R_\alpha \mu_n \uparrow R_\alpha \mu\) q.e.,  taking \(n \to \infty\), we have \(\|R_\alpha \mu\|_{\infty, q} \leq \frac{1}{1 - e^{-\alpha t}} \|\mathbb{E}_{\cdot}[A_t^\mu]\|_{\infty}\). Finally, by letting \(\alpha \to \infty\) and then \(t \to 0\), we obtain \(\lim_{\alpha \to \infty} \|R_\alpha \mu\|_{\infty, q} = 0\), which implies $\mu \in \mathcal{K}$.
\end{proof}

\begin{remark}
  \label{Rem:Kato}
  While there are several equivalent definitions of the Kato class,
  it is often defined in the sense of positive continuous additive functionals
  as in the lemma above.
  In the latter half of the proof of Lemma \ref{lemdefKato},
  we take an appropriate nest to ensure the boundedness of
  the $\alpha$-potential, but in fact, it is not necessary to take such a nest.
  Assuming the short-time asymptotic of the PCAF $A_{t}$,
  we can see that $\mathbb{E}_{x}[A_{t}^{\mu}]$ has at most linear growth in $t$.
  Using integration by parts, we can obtain the boundedness of $\alpha$-potential
  ($\alpha > 0$). The details are omitted here since this topic is not essential. 
\end{remark}

We consider relationships between the several classes of \(\mathcal{S}\). The following are the main results of this section.

\begin{theorem}\label{Thm_loc}
It holds that
\begin{align*}
\mathcal{S}_R\ =\ (\mathcal{S}_R)_{loc}\ = \ (\mathcal{S}_F)_{loc}\ \supset \ (\mathcal{S}_0)_{loc}\ &\supset\ (\mathcal{S}_{00})_{loc}\ =\ \mathcal{D}_{loc}\  =\ (\mathcal{D}_0)_{loc}\ =\ (\mathcal{D}_{00})_{loc}\\&\ \ \supset\ \mathcal{K}_{loc}\ =\ (\mathcal{K}_0)_{loc}\ =\ (\mathcal{K}_{00})_{loc}\ =\ (\mathcal{K}_{\infty})_{loc}.
\end{align*}
\end{theorem}

\begin{theorem}\label{Thm_col}
It holds that
\begin{align*}
\mathcal{S}\ =\ (\mathcal{S}_R)_{col}\ = \ (\mathcal{S}_F)_{col}\ =\ (\mathcal{S}_0)_{col}\ &=\ (\mathcal{S}_{00})_{col}\ =\ \mathcal{D}_{col}\  =\ (\mathcal{D}_0)_{col}\ =\ (\mathcal{D}_{00})_{col}\\
&\ \ =\ \mathcal{K}_{col}\ =\ (\mathcal{K}_0)_{col}\ =\ (\mathcal{K}_{00})_{col}\ =\ (\mathcal{K}_{\infty})_{col}.
\end{align*}
\end{theorem}

Theorem \ref{Thm_col} says that, for any positive Borel measure \(\mu\) on \(E\), \(\mu \in \mathcal{S}\) if and only if there exists a compact nest \(\{F_\ell\}_\ell\) such that \(1_{F_\ell}\mu \in \mathcal{T}\) for each \(\ell\) and \(\mu(E\setminus \bigcup F_\ell)=0\), where \(\mathcal{T}\) is any of \(\mathcal{S}_R, \mathcal{S}_F, \mathcal{S}_0, \mathcal{S}_{00}, \mathcal{D}, \mathcal{D}_0, \mathcal{D}_{00}, \mathcal{K}, \mathcal{K}_0, \mathcal{K}_{00}, \mathcal{K}_\infty.\)

To prove Theorem \ref{Thm_loc}, \ref{Thm_col}, we need the following lemma essentially due to \cite[Lemma 3.7]{ABM91}.
\begin{lemma}[{cf. \cite[Lemma 3.7]{ABM91}}] \label{K00isS00}
It holds that $\mathcal{K}_{00} =\mathcal{K} \cap \mathcal{S}_F $ and  $\mathcal{D}_{00} =\mathcal{D} \cap \mathcal{S}_F $.
\end{lemma}

\begin{proof}
Since  \(\mathcal{K}_{00} := \mathcal{K} \cap \mathcal{S}_{00}\) and \(\mathcal{D}_{00} := \mathcal{D} \cap \mathcal{S}_{00}\), it is clear that $\mathcal{K}_{00} \subset \mathcal{K} \cap \mathcal{S}_F $ and  $\mathcal{D}_{00} \subset \mathcal{D} \cap \mathcal{S}_F $.

We take \(\mu \in  \mathcal{D} \cap \mathcal{S}_F \).  Since $\mu$ has finite mass and bounded potential, it is enough to show that \(R_1\mu \in \mathcal{F}\). By the Revuz correspondence \cite[Theorem 4.1.1 (iii)]{CF12}, we have
\begin{eqnarray*}
\int_E |R_1\mu|^2\, dm &=& \int_E R_1\mu \cdot R_1\mu\, dm\ =\ \mathbb{E}_{R_1\mu\,m}\left[\int_0^\infty e^{- t}\,dA_t^\mu \right]\\
&=& \int_E R_1(R_1\mu)\, d\mu \ \leq \ \|R_1\mu\|_{\infty, q} \, \mu(E) \ <\  \infty,
\end{eqnarray*}
which implies that $R_1\mu \in L^2(E;m)$. In the same way as in the proof of Lemma \ref{lemdefKato}, we have 
\[R_1\mu-e^{-t}P_tR_1\mu = \mathbb{E}_x\left[\int_0^t e^{-s} dA_s^\mu \right]\]
and so, by the Revuz correspondence,
\[\lim_{t \to 0} \frac{1}{t} (R_1\mu-e^{-t}P_tR_1\mu, R_1\mu) = \int_E R_1\mu\,d\mu \leq \|R_1 \mu\|_{\infty,q} \, \mu(E) < \infty. \]
Since the limit of the approximating energy forms is finite, we conclude that \(R_1\mu \in \mathcal{F}\), and hence \(\mu \in \mathcal{S}_{00}\).
\end{proof}

The following is called the Stollmann--Voigt inequality. Under the assumption of the existence of the heat kernel, this inequality is well-known for a smooth measure in the strict sense (see, e.g., \cite{SV96, ST05}), and this inequality is proved in \cite{BA04} without such an assumption.  Here, we present an alternative probabilistic proof, also without relying on the heat kernel assumptions.
  Note that this approach yields the sharp bound without the constant $4$ that would appear if one used the capacitary strong type inequality.

\begin{proposition}[the Stollmann--Voigt inequality] \label{SVineq}
For any \(\mu \in \mathcal{S}\) and \(f\in \mathcal{F}\), we have 
\[\int_E f^2 \, d\mu \leq \|R_1\mu\|_\infty\,\mathcal{E}_1(f,f) \le\infty.\]
\end{proposition}

\begin{proof} 
By using the monotone convergence theorem, it is enough to show the inequality for a bounded and nonnegative $f\in \dom$ and $\mu \in {\cal S}_{00}$. For $\beta>0$, we set $f_\beta:=\beta R_{\beta+1} f$. It is known that $f_\beta \to f$ with respect to both the $\form_1$-norm and the $L^2(\mu)$-norm. 

Before estimating the specific norms, we first establish a general pointwise inequality. 
For any $\nu \in {\cal S}_{00}$ and a nonnegative Borel function $h \ge 0$, the potential $R_1(h \nu)$ can be pointwise bounded by the Cauchy-Schwarz inequality via its probabilistic representation:
\begin{align}
\big(R_1(h \nu)(x)\big)^2
& =\Big({\mathbb E}_x\Big[ \int_0^\infty e^{-s} h(X_s)dA_s^\nu\Big]\Big)^2  \nonumber \\
& \le {\mathbb E}_x\Big[ \int_0^\infty e^{-s} h(X_s)^2dA_s^\nu\Big] \, {\mathbb E}_x\Big[ \int_0^\infty e^{-s} dA_s^\nu\Big]  \nonumber \\
& =R_1(h^2\nu)(x) \, U_1\nu(x). \label{eq:potential_CS}
\end{align}
Note that \(R_1\nu =  U_1\nu\) holds for \(\nu \in \mathcal{S}_{00}\). Now, noting that $R_1(\beta(f-\beta R_{\beta+1}f))=\beta R_{\beta+1}f = f_\beta$, we set $g_\beta:=\beta(f-f_\beta) \ge 0$ so that $f_\beta = R_1g_\beta$. To estimate the $L^2(\mu)$-norm of $R_1g_\beta$, we use the duality expression:
\begin{align*}
\int_E f_\beta^2 d\mu = \int_E (R_1g_\beta)^2 d\mu 
& = \sup_{\|\phi\|_{L^2(\mu)}\le 1} \Big( \int_E R_1g_\beta \, \phi \, d\mu \Big)^2 \\
&\le  \sup_{\|\phi\|_{L^2(\mu)}\le 1} \Big( \int_E R_1g_\beta \, |\phi| \, d\mu \Big)^2 \\
& = \sup_{\|\phi\|_{L^2(\mu)}\le 1} \Big( \int_E g_\beta  R_1(|\phi| \mu) \, dm \Big)^2,
\end{align*}
where we used the Revuz correspondence in the last equality. By the Cauchy-Schwarz inequality for the $\form_1$-inner product and the monotone convergence theorem, we have
\begin{align*}
\Big(\int_E g_\beta  R_1(|\phi| \mu) \, dm\Big)^2 &= \lim_{n\to \infty} \Big(\int_E g_\beta  R_1((|\phi|\wedge n) \mu) \, dm\Big)^2 \\
& = \lim_{n\to \infty} \form_1(R_1g_\beta,  U_1((|\phi|\wedge n)\mu))^2 \\
&\le \lim_{n\to \infty} \form_1(R_1g_\beta, R_1g_\beta)  \form_1 (U_1((|\phi|\wedge n) \mu),U_1((|\phi|\wedge n) \mu)) \\
& = \lim_{n\to \infty} \Big(\int_E g_\beta R_1g_\beta \, dm \Big) \Big( \int_E U_1((|\phi|\wedge n) \mu) |\phi| \, d\mu\Big)\\
& =\Big(\int_E g_\beta R_1g_\beta \, dm \Big) \Big( \int_E R_1(|\phi| \mu) |\phi| \, d\mu\Big).
\end{align*}
For the second integral on the right-hand side, applying the general inequality \eqref{eq:potential_CS} with $h=|\phi|$ and $\nu=\mu$, we can further estimate it as follows:
\begin{align*}
\int_E R_1(|\phi| \mu) |\phi| \, d\mu  
& \le \Big(\int_E \big(R_1(|\phi|\mu)\big)^2 d\mu\Big)^{1/2}  \Big(\int_E\phi^2 d\mu\Big)^{1/2} \\
& \le \Big(\int_E R_1(\phi^2 \mu) \,  U_1\mu \, d\mu\Big)^{1/2} \cdot 1  \\
& \le \|U_1\mu\|_\infty^{1/2} \Big( \int_E R_1(\phi^2\mu) \, d\mu\Big)^{1/2} \\
&= \|U_1\mu\|_\infty^{1/2}   \Big(\int_E U_1\mu \  \phi^2  d\mu\Big)^{1/2} \\
&\le  \|U_1\mu\|_\infty \|\phi\|_{L^2(\mu)}\  \le \  \|U_1\mu\|_\infty. 
\end{align*}
Therefore we find that
\begin{align}
\label{eq:2}
\int_E f_\beta^2 d\mu \le \|U_1\mu\|_\infty \int_E g_\beta R_1 g_\beta \, dm
\end{align}
holds for any $\beta>0$. Since $f_\beta=\beta R_{\beta+1}f  \to f$ in $L^2(\mu)$
as $\beta \to \infty$, the left-hand side converges to $\int_E f^2 d\mu$.
On the other hand, the integral term on the right-hand side is equal to
$\form_1(\beta R_{\beta+1}f, \beta R_{\beta+1}f)$,
which converges to $\mathcal{E}_1(f, f)$ by the spectral representation. Consequently, we conclude that
\begin{align}
\label{eq:6}
\int_E f^2 d\mu \le \|U_1 \mu\|_\infty \mathcal{E}_1(f, f).
\end{align}
This proof is complete.
\end{proof}

\begin{proof}[Proof of Theorem \ref{Thm_loc}]
For any \(\mu \in (\mathcal{S}_R)_{loc}\) and a compact set \(K\),  \(1_K\mu \in \mathcal{S}_R\) holds and so \(\mu(K)=1_K\mu(K)<\infty\) and \(\mu\) is smooth. Then we take a compact nest \(\{F_k\}_k\) such that \(\mu(\cap_k F_k^c)=0\) and \(1_{F_k}\mu \in \mathcal{S}_0.\) For any Borel set \(A\) with \(\mu(A)<\infty\), we take \(k\) such that \(\mu(A)-\mu(A\cap F_k) \leq \varepsilon\). By using the inner regularity for \(1_{F_k}\mu \in \mathcal{S}_R\), we can take a compact set \(K \subset A\cap F_k\) such that \(\mu(A\cap F_k)-\mu(K) \leq \varepsilon.\) Then we have \(\mu(A)-\mu(K)\leq 2\varepsilon\). For a Borel set \(A\) with \(\mu(A)=\infty\) and any \(M>0\), we take large \(k\) such that \(\mu(A \cap F_k)\geq 2M\) and we can choose a compact set \(K\subset A\cap F_k\) with \(\mu(K)\geq M\). Hence the inner regularity holds for \(\mu\) and so \((\mathcal{S}_R)_{loc} \subset \mathcal{S}_R\) holds. Clearly \((\mathcal{S}_R)_{loc} \supset \mathcal{S}_R\) holds, so \((\mathcal{S}_R)_{loc} = \mathcal{S}_R\).

Since \(\mathcal{D}\supset \mathcal{D}_0 \supset \mathcal{S}_{00} \supset \mathcal{D}_{00}\), it holds that \(\mathcal{D}_{loc}\supset (\mathcal{D}_0)_{loc} \supset (\mathcal{S}_{00})_{loc}\supset (\mathcal{D}_{00})_{loc}\) follows from Proposition \ref{Prop_loccol_1}. Since \(\mathcal{D}_{00} = \mathcal{S}_{00} \cap \mathcal{D}\), it holds that \((\mathcal{D}_{00})_{loc} = (\mathcal{S}_{00})_{loc} \cap \mathcal{D}_{loc}\) and so \((\mathcal{S}_{00})_{loc} = (\mathcal{D}_{00})_{loc}\).

Since a finite measure on a locally compact separable metric space is Radon, \(\mathcal{S}_R \supset \mathcal{S}_F\) holds and so \((\mathcal{S}_R)_{loc} \supset (\mathcal{S}_F)_{loc}\). For any compact set \(K\) and \(\mu \in \mathcal{S}_R\), \(1_K \mu \in \mathcal{S}_R\) is a finite measure, so  \(\mu \in (\mathcal{S}_F)_{loc}\), and  \((\mathcal{S}_R)_{loc} = (\mathcal{S}_F)_{loc}\) holds. Since \(\mathcal{S}_R \supset \mathcal{S}_0 \supset \mathcal{S}_{00}\), it holds that \((\mathcal{S}_R)_{loc} \supset (\mathcal{S}_0)_{loc} \supset (\mathcal{S}_{00})_{loc}\).

For \(\mu \in (\mathcal{D}_0)_{loc}\) and any compact set \(K\), since \(1_K \mu \in \mathcal{D}\cap \mathcal{S}_0\) holds, by taking \(f \in \mathcal{F}\cap C_c\) satisfying \(f=1\) on \(K\), we have \(1_K\mu(E)=\mu(K) \leq \int_K f d\mu \leq C_K \mathcal{E}_1(f,f) <\infty.\) Hence \(1_K\mu \in \mathcal{S}_{00}\) and we have \((\mathcal{S}_{00})_{loc}\supset (\mathcal{D}_0)_{loc}\). For \(\mu \in \mathcal{D}_{loc}\) and any compact set \(K\), by the Stollmann-Voigt inequality (Proposition \ref{SVineq}), we have 
\[\int_K |f|d\mu \leq \sqrt{\mu(K)} \sqrt{\int_K |f|^2d\mu} \leq  \sqrt{\mu(K) \|R_1(1_K\mu)\|_\infty }\sqrt{\mathcal{E}_1(f,f)} \]
for any \(f\in \mathcal{F}\cap C_c\), and so \(\mathcal{D}_{loc} \subset (\mathcal{S}_{00})_{loc}.\) Hence \((\mathcal{S}_{00})_{loc} \supset \mathcal{D}_{loc} \supset (\mathcal{D}_0)_{loc} \supset (\mathcal{D}_{00})_{loc} = (\mathcal{S}_{00})_{loc}.\)

It is clear that \(\mathcal{D}_{loc} \supset \mathcal{K}_{loc}\supset(\mathcal{K}_0)_{loc}\supset(\mathcal{K}_{00})_{loc}\). For any \(\mu \in \mathcal{K}_{loc}\) and a compact set \(K\), it holds that \(1_K\mu \in \mathcal{K}\). Since \(\mu\) is a Radon measure by the Stollmann-Voigt inequality (Proposition \ref{SVineq}), \(1_K\mu\) is a finite measure and so \(1_K\mu \in \mathcal{S}_{00}\) by Lemma \ref{K00isS00}. Hence \(\mathcal{K}_{loc}\subset(\mathcal{K}_{00})_{loc}\) holds.

Since \(\mathcal{K}_\infty \subset \mathcal{K}\), we have \((\mathcal{K}_\infty)_{loc} \subset \mathcal{K}_{loc}\). For any \(\mu \in \mathcal{K}_{loc}\) and any compact set \(K\), it holds that \(\|R_1(1_{K^c}1_K\mu)\|_\infty = 0\) and so \(1_K\mu \in \mathcal{K}_\infty\). Hence we have  \((\mathcal{K}_\infty)_{loc} = \mathcal{K}_{loc}\).
\end{proof}

We call an increasing sequence \(\{F_k\}_k\) of closed sets {\it a strong nest} if
$$
\mathbb{P}_x \left( \lim_{k \to \infty} \sigma_{E \setminus F_k} < \infty \right) = 0 \quad \text{for q.e. } x \in E.
$$
Note that a strong nest is a nest (see \cite[pp. 94--95]{CF12}).  The following lemma is used in the proof of \cite[Theorem 2.4]{AM92}. Here, we provide a proof using a different approach based on the Fukushima decomposition.

\begin{lemma}\label{S00isalmostK00}
For \(\mu \in \mathcal{S}_{00}\), \(\mathbb{E}_x[A_t^\mu]\) converges to \(0\) quasi-uniformly as \(t \to 0\). That is,  there exists a strong nest \(\{G_k\}_k\) such that \(\mathbb{E}_x[A_t^\mu]\) converges to \(0\) uniformly on each set \(G_k\).
\end{lemma}

\begin{proof}
Take \(\mu \in \mathcal{S}_{00}\). Applying the Fukushima decomposition to  the potential  $U_1\mu$ (cf. \cite[Lemma 5.4.1]{FOT11}), there exists a martingale additive functional \(M_t^{[U_1\mu]}\) such that
\[U_1\mu(X_t)-U_1\mu(x) = M_t^{[U_1\mu]} +\int_0^t  U_1\mu(X_s)ds -A_t^\mu, \ \ \mathbb{P}_x\textrm{-a.s.,\ for\ q.e.}  \ x.\]
Here, $U_1\mu$  is a quasi-continuous version of \(U_1\mu\). By taking an expectation  with respect to $\mathbb{P}_x$ , we have
\begin{equation} \label{eq:FUkudeco}
P_t U_1\mu(x)-U_1\mu(x) = \int_0^t P_s U_1\mu(x)ds -\mathbb{E}_x[A_t^\mu] \ \ \textrm{for\ q.e.} \ x. 
\end{equation}
By \cite[Proposition 3.1.9]{CF12}, \(P_s U_1\mu\) is also quasi-continuous for each $s$ and noting the equality of the $m$-essential sup norm and the quasi-sup norm, we  have
$$
\left| \int_0^t P_s U_1\mu (x)ds \right| \leq t \|U_1\mu \|_\infty \textrm{\ \ q.e.}  \ x.
$$
Thus, \(\int_0^t P_s U_1\mu(x) ds\) converges to \(0\) uniformly on \(E\) (except on a set of zero capacity) as \(t \downarrow 0\).

On the other hand, since \(P_t U_1\mu\) converges to \(U_1\mu\) with respect to $\sqrt{\mathcal{E}_1}$-norm as \(t \downarrow 0\), there exist a sequence \(t_m \downarrow 0\) and a strong nest \(\{G_k\}_k\) such that $P_{t_m} U_1\mu$ converges to $U_1\mu$ uniformly on each \(G_k\) (cf. \cite[Theorem 3.5.4]{CF12}). 
Combining this with \eqref{eq:FUkudeco}, we see that \(\mathbb{E}_x[A_{t_m}^\mu]\) converges to \(0\) uniformly on each \(G_k\) as \(m \to \infty\). Since the PCAF \(A_t^\mu\) is non-decreasing in \(t\), the expectation \(\mathbb{E}_x[A_t^\mu]\) also converges to \(0\) uniformly on each \(G_k\) for the continuous parameter \(t\).
\end{proof}
\begin{proof}[Proof of Theorem \ref{Thm_col}]
By \cite[Theorem 2.2.4]{FOT11} and \cite[Theorem 2.3.15]{CF12}, \(\mathcal{S}=(\mathcal{S}_0)_{col} = (\mathcal{S}_{00})_{col}\). Since \(\mathcal{S}_R \supset \mathcal{S}_0\) and \(\mathcal{S}_F \supset \mathcal{S}_{00}\), \(\mathcal{S}=(\mathcal{S}_R)_{col} = (\mathcal{S}_F)_{col}\) follows from Proposition \ref{Prop_loccol_1}.

While the identity $\mathcal{K}_{loc} = (\mathcal{K}_0)_{loc} = \mathcal{S}$ is essentially proved in \cite[Theorem 3.3]{ABM91}, we provide a proof here for the reader's convenience.

For \(\mu \in \mathcal{S} = (\mathcal{S}_{00})_{col}\), we take \(\{F_n\}_n\) satisfying \(\mu \in \mathcal{S}_{00}(\{F_n\})\). By Lemma \ref{S00isalmostK00}, for each \(n\), we take a strong nest \(\{G_k^{(n)}\}\) such that \(\mathbb{E}_x[A_t^{1_{F_n} \mu}]\) converges to \(0\) uniformly on each set \(G_k^{(n)}\) as \(t\to 0\). 
 For each $n \in \mathbb{N}$, we define $G_n := \bigcup_{k=1}^n (F_k \cap G_k^{(n)})$. Then, each $G_n$ is a compact set and $\{G_n\}$ forms a nest satisfying \(\mu(E\setminus \bigcup_{n=1}^\infty G_n)=0\). Moreover, we have 
\[\lim_{t \to 0} \sup_{x\in G_n} \mathbb{E}_x[A^{1_{G_n}\mu} _t] \leq \lim_{t \to 0} \sup_{x\in F_n} \mathbb{E}_x[ A^{1_{F_n}\mu}_t] =0\]
and, by the strong Markov property, for any \(x\in G_n^c\),
\begin{eqnarray*}
\mathbb{E}_x[ A^{1_{G_n}\mu}_t] &=& \mathbb{E}_x\left[1_{\{\sigma(G_n)<t\}}\ \int_{\sigma(G_n)}^t 1_{G_n}(X_s) dA^\mu _t \right] \\
&=&\mathbb{E}_x\left[1_{\{\sigma(G_n)<t\}}\ \mathbb{E}_{X_{\sigma(G_n)}}[A^{1_{G_n}\mu}_t]\right]\\
& \leq &\sup_{x\in G_n} \mathbb{E}_x[A^{1_{G_n}\mu}_t],
\end{eqnarray*}
where \(\sigma(G_n):=\inf\{t>0 : X_t \in G_n\}\) is the first hitting time to \(G_n\).
Hence we have \[\lim_{t \to 0} \sup_{x\in E} \mathbb{E}_x[A^{1_{G_n}\mu} _t]  =0\]
and so \(\mu \in \mathcal{K}_{00}\). This means that \(\mathcal{S}\subset (\mathcal{K}_{00})_{col}\). Since \(\mathcal{K}_{00} \subset \mathcal{S}_{00}\), by Proposition \ref{Prop_loccol_1}, \(\mathcal{S} = (\mathcal{K}_{00})_{col}\) holds. Since \(\mathcal{K}_{00} \subset \mathcal{K}_0 \subset \mathcal{K}\) and \(\mathcal{K}_{00} \subset \mathcal{D}_{00} \subset \mathcal{D}_0 \subset \mathcal{D}\), it holds that \(\mathcal{S}=\mathcal{D}_{col}=(\mathcal{D}_0)_{col}=(\mathcal{D}_{00})_{col}=\mathcal{K}_{col}= (\mathcal{K}_0)_{col}=(\mathcal{K}_{00})_{col}\).

By Proposition \ref{Prop_loccol_3} and Theorem \ref{Thm_loc}, we have
\[(\mathcal{K}_\infty)_{col}=((\mathcal{K}_\infty)_{loc})_{col}=(\mathcal{K}_{loc})_{col}=(\mathcal{K}_{col})_{loc}=\mathcal{S}_{loc}=\mathcal{S}.\]
\end{proof}

At the end of this section, we introduce subclasses \(\mathcal{T}\) satisfying \(\mathcal{S} \not = \mathcal{T}_{col}\) in general, as shown in the following propositions.

Let \(\mathcal{S}_{AC}\) be the set of all positive Borel measures charging no set of zero capacity, which are absolutely continuous with respect to \(m\). 
\begin{proposition}
It holds that \((\mathcal{S}_{AC})_{loc} = (\mathcal{S}_{AC})_{col} = \mathcal{S}_{AC}\).
\end{proposition}
\begin{proof}
For any \(\mu \in (\mathcal{S}_{AC})_{col}\), we take some compact nest \(\{F_k\}_k\) such that \(1_{F_k}\mu \in \mathcal{S}_{AC}\). For any set \(A\) with \(m(A)=0\), we have 
\[\mu(A)= \lim_{k\to \infty} (1_{F_k}\mu)(A) = 0\]
and so \(\mathcal{S}_{AC} = (\mathcal{S}_{AC})_{col}\). By Proposition \ref{Prop_loccol_3}, \(\mathcal{S}_{AC} = (\mathcal{S}_{AC})_{loc}\) holds.
\end{proof}

Let \(\mathcal{S}_{SG}\) be the set of all positive Borel measures charging no set of zero capacity, which are singular with respect to \(m\). 
\begin{proposition}
It holds that \((\mathcal{S}_{SG})_{loc} = (\mathcal{S}_{SG})_{col} = \mathcal{S}_{SG}\).
\end{proposition}
\begin{proof}
For any \(\mu \in (\mathcal{S}_{SG})_{col}\), we take some compact nest \(\{F_k\}_k\) such that \(1_{F_k}\mu \in \mathcal{S}_{SG}\) for each $k$. There exists a sequence of sets \(\{A_k\}_k\) with \(m(A_k)=0\) and \(\mu(F_k\setminus A_k)=0\). We set \(A:=\bigcup_{k=1}^\infty A_k\) then we have
\[m(A)\leq \sum_{k=1}^\infty m(A_k) = 0\]
and 
\[\mu(A^c) \leq \varlimsup_{k\to \infty} \mu(A_k^c) \leq  \varlimsup_{k\to \infty} \mu(F_k \setminus A_k) + \varlimsup_{k\to \infty} \mu(F_k^c) = 0,\]
and so \(\mathcal{S}_{SG} = (\mathcal{S}_{SG})_{col}\). By Proposition \ref{Prop_loccol_3}, \(\mathcal{S}_{SG} = (\mathcal{S}_{SG})_{loc}\) holds.
\end{proof}

The above two propositions mean that, in general, the set of all smooth measures cannot be approximated by only one of the spaces of smooth measures absolutely continuous or singular with respect to the underlying measure. Since \(m \in \mathcal{S}_{AC}\), \(\mathcal{S}_{AC}\) is not empty set. However, $\mathcal{S}_{SG}$ may be empty, hence $\mathcal{S}_{AC} = \mathcal{S}$. This occurs, for example, in the case of a process on a graph.

Denote by $L^0(E;m)$ the space of all $m$-measurable functions on $E$ up to $m$-equivalence. For a linear subspace $L$ of $L^0(E;m)$, we define
 \[\mathcal{S}_{L} := \{\mu \in \mathcal{S}_{AC} : \text{there\ exists\ }f\in L \text{\ such\ that\ }d\mu =fdm.\} \]
and \(L_{loc} := \{f \in L^0(E;m) : 1_Kf \in L \text{\ for\ any\ compact\ set\ }K\}.\)

\begin{proposition}
Suppose that $1_K f \in L_{loc}$ for each $f \in L$ and each compact set $K$. Then, the identity $(\mathcal{S}_L)_{loc} = \mathcal{S}_{L_{loc}}$ holds.
\end{proposition}
\begin{proof}
For any \(\mu = fm \in \mathcal{S}_{L_{loc}}\) with \(f\in L_{loc}(E;m)\) and any compact set \(K\), we can write \(1_K\mu = 1_K f m\) and \(1_K f \in L\) and so \(\mu \in (\mathcal{S}_{L})_{loc}.\) Conversely, for any \(\mu \in (\mathcal{S}_{L})_{loc}\) with \(\mu = fm\) and any compact set \(K\), we have \(1_Kfm \in \mathcal{S}_{L}\) and, by the definition, \(1_Kf \in L\) and so \(f \in L_{loc}\).
\end{proof}

In general, it holds that \(\mathcal{S}_{L_{loc}} \subsetneqq  (\mathcal{S}_{L})_{col}\). Indeed, for Brownian motion on \(\mathbb{R}^d\), a sequence \(F_k:=\{x : 1/k \leq |x|\leq k\}\) is a compact nest and \(|x|^{-\beta}dx \in \mathcal{S}_{L^p}(\{F_k\})\) but \(|x|^{-\beta}1_{\{|x|\leq 1\}} \not \in L^p\) for \(p \geq d/\beta.\)

\section{Smooth Radon measures and comparison with the Kato class}\label{sec_Radon_Kato}

While smooth measures are fundamental in Dirichlet form theory, their potentials do not necessarily belong to $L^2$, nor does $L^\infty$-potential ensure the Radon property of the measure in general. Our goal in this section is to establish a sufficient condition for a smooth measure to be Radon by analyzing its potential using the regularity of the Dirichlet form.

Furthermore, we investigate the properties of measures in the Kato class. Since the Kato class is defined by a stronger boundedness condition on the potential, it forms a significant subclass of smooth measures.

To achieve this, we first derive an inequality that bounds the measure of a compact set by its potential, thereby establishing the Radon property.
  \begin{lemma}
    \label{Lem:smooth-01}
    Assume that $\mu\in{\cal S}$ and its $\alpha$-potential
    $R_\alpha\mu$ belongs to $L^\infty(E;m)$ for  all $\alpha>0$. 
    Let $\{F_\ell\}$ be a (compact) nest attached 
    to $\mu$, namely, $1_{F_\ell}\mu \in {\cal S}_0$ for all $\ell \in{\mathbb N}$.
    Then for any compact set $K$ and $\alpha>0$,
    there exists a $C=C(\alpha, K)>0$ such that
    the following hold for all $\ell \in{\mathbb N}:$ 
    \begin{equation} \label{eqn:est1}
      \mu(K\cap F_\ell) \le
      C \sqrt{\form_\alpha(U_\alpha(1_{F_\ell \cap K}\mu), U_\alpha(1_{F_\ell \cap K}\mu))}
    \end{equation}
    and 
    \begin{equation} \label{eqn:est2}
      \form_\alpha(U_\alpha(1_{F_\ell \cap K}\mu), U_\alpha(1_{F_\ell \cap K}\mu))
      \le \Big(1+\frac{|\alpha-1|}{\alpha} \Big) \|R_1\mu\|_{\infty} \mu(K \cap F_\ell).
\end{equation}
\end{lemma}
\noindent
{\it Proof.} We first show \eqref{eqn:est1}. Take any compact set $K\subset E$. By using the regularity of $(\form, \dom)$, we can take a function 
$\varphi \in \dom\cap C_0(E)$ such that $0\le \varphi \le 1$ and $\varphi=1$ on $K$. 
Setting $C = C(\alpha, K) := \sqrt{\mathcal{E}_\alpha(\varphi, \varphi)}$, we have for each $\ell \in \mathbb{N}$:
\begin{align*}
\mu(K \cap F_\ell) &\le \int_{K \cap F_\ell} \varphi(x) \mu(dx) = \mathcal{E}_\alpha(\varphi, U_\alpha(1_{F_\ell \cap K}\mu)) \\
&\le \sqrt{\mathcal{E}_\alpha(\varphi, \varphi)} \sqrt{\mathcal{E}_\alpha(U_\alpha(1_{F_\ell \cap K}\mu), U_\alpha(1_{F_\ell \cap K}\mu))} \\
&= C \sqrt{\mathcal{E}_\alpha(U_\alpha(1_{F_\ell \cap K}\mu), U_\alpha(1_{F_\ell \cap K}\mu))}.
\end{align*}
This proves \eqref{eqn:est1}.

Next, \eqref{eqn:est2} follows from the resolvent-type equation:
$$
U_\alpha \nu=U_1\nu -(\alpha-1) R_\alpha U_1 \nu \quad {\rm for} \quad \nu \in {\cal S}_0,
$$
together with the contraction property of the resolvent $\{R_\alpha\}$. Specifically, for $\nu=1_{K\cap F_\ell} \mu$, we have 
\begin{align*}
\form_\alpha(U_\alpha \nu, U_\alpha\nu) 
& =\int_{K\cap F_\ell} \widetilde{U_\alpha \nu} d\mu \\
& \le \Big( \|U_1\nu \|_\infty + |\alpha-1| \,  \| R_\alpha U_1\nu \|_\infty \Big)  \mu(K\cap F_\ell) \\
& = \Big(\|U_1\nu \|_\infty +\frac{|\alpha-1|}{\alpha} \|\alpha R_\alpha U_1\nu\|_\infty \Big)   \mu(K\cap F_\ell) \\
& \le \Big(1+\frac{|\alpha-1|}{\alpha}\Big) \|U_1(1_{K\cap F_\ell} \mu) \|_\infty  \mu(K\cap F_\ell) \\
&\le \left( 1 + \frac{|\alpha - 1|}{\alpha} \right) \left\| R_{1} \mu \right\|_{\infty} \mu(K \cap F_{\ell}),
\end{align*}
where the last inequality follows from \cite[Proposition 3.10]{NTTU25}. 
Hence we obtain \eqref{eqn:est2}. 
\hfill \fbox{}

\begin{proposition} \label{Prop:smooth}
  Suppose that $\mu \in \mathcal{S}$ satisfies the same assumptions as in 
  Lemma \ref{Lem:smooth-01}.
  Then $\mu$ is a Radon measure.
\end{proposition}
\begin{proof}
Combining \eqref{eqn:est1} and \eqref{eqn:est2}, we see that for each \(\ell \in \mathbb{N}\),
\[
\mu(K \cap F_\ell) \leq C^2 \left(1 + \frac{|\alpha - 1|}{\alpha} \right) \|R_1 \mu\|_{\infty}.
\]
Since the right-hand side is independent of \(\ell\), by letting \(\ell \to \infty\), we obtain
\[
\mu(K) = \lim_{\ell \to \infty} \mu(K \cap F_\ell) \leq C^2 \left(1 + \frac{|\alpha - 1|}{\alpha} \right) \|R_1 \mu\|_{\infty} < \infty.
\]
Consequently, \(\mu\) assigns finite mass to every compact set \(K \subset E\), which implies that \(\mu\) is a Radon measure.
\end{proof}

\begin{corollary}
\label{Cor:FEI}
  Assume that $\mu\in{\cal S}$ and its $\alpha$-potential $R_\alpha\mu$ belongs to $L^\infty(E;m)$ for  all $\alpha>0$. 
Then for any compact set $K\subset E$, the restricted measure $1_K\mu$ is of finite energy integral; that is, $1_K\mu \in \mathcal{S}_0$.

In particular, if $\mu(E) < \infty$, then $\mu$ itself belongs to $\mathcal{S}_0$, which further implies that $\mu \in \mathcal{S}_{00}$.
\end{corollary}

\begin{remark}  \label{rem:3.4}
\begin{itemize}
\item[\sf (1)] By considering the case $\alpha=1$ in Lemma \ref{Lem:smooth-01} and taking the infimum over all $\varphi \in \dom \cap C_0(E)$ satisfying $0\le \varphi \le 1$ and $\varphi=1$ on a compact set $K$, then we can take ${\rm Cap}(K)$ as the constant $C=C(K, 1)>0:$ for any $\ell$,
$$
\left\{
\begin{array}{cl} 
\mu(K\cap F_\ell) & \le \dis \sqrt{{\rm Cap}(K)} \sqrt{\form_1(U_1(1_{K\cap F_\ell}\mu), U_1(1_{K\cap F_\ell}\mu))}, \\[5pt]
\mu(K) & \le {\rm Cap}(K) \|R_1\mu\|_\infty.
\end{array}
\right.
$$
{
  Moreover, for any $\mu, \nu \in \mathcal{S}_{00}$ and any compact set $K \subset E$, it holds that
$$ 
\nu(K) \le \mathrm{Cap}(K) \|U_1(1_K\nu)\|_\infty
$$
and
$$
\form_1(U_1(1_K\mu)-U_1(1_K\nu), U_1(1_K\mu)-U_1(1_K\nu))
\le \big( \mu(K)+ \nu(K) \big) \big\| U_1(1_K\mu)-U_1(1_K\nu) \big\|_{\infty}.
$$
Indeed, letting $e_K \in \mathcal{F}$ be the 1-equilibrium potential of $K$, we have 
$$
\nu(K) \le \int_K e_K d\nu = \mathcal{E}_1(e_K, U_1(1_K\nu)) \le \sqrt{\mathrm{Cap}(K)} \sqrt{\mathcal{E}_1(U_1(1_K\nu), U_1(1_K\nu))}.
$$
Since $\mathcal{E}_1(U_1(1_K\nu), U_1(1_K\nu)) = \int_K U_1(1_K\nu) d\nu \le \|U_1(1_K\nu)\|_\infty \nu(K)$, these yield the first inequality. The second one readily follows from 
$$
\form_1(U_1(1_K\mu)-U_1(1_K\nu), U_1(1_K\mu)-U_1(1_K\nu)) = \int_K (U_1(1_K\mu)-U_1(1_K\nu)) d(\mu-\nu) 
$$
by bounding the integral with $\big\| U_1(1_K\mu)-U_1(1_K\nu) \big\|_{\infty} |\mu-\nu|(K)$ and noting $|\mu-\nu|(K) \le \mu(K) + \nu(K)$.}

\item[\sf (2)] It is worth noting that, in general, the boundedness of the potential ($R_\alpha \mu \in L^\infty$) 
and the finiteness of the measure ($\mu(E) < \infty$) do not automatically imply that the potential belongs to $L^2(E; m)$ 
unless $\mu$ is assumed to be a smooth measure. In our setting, the smoothness of $\mu$ is essential to ensure that the energy 
integral $\int_E R_\alpha \mu \, d\mu$ is well-defined and finite, thereby establishing $\mu \in \mathcal{S}_0$.
\end{itemize}
\end{remark}

\noindent
{\it Proof of Corollary \ref{Cor:FEI}}.  From Proposition \ref{Prop:smooth}, the measure $1_K \mu$ has finite total mass $\mu(K) < \infty$ 
for a compact set $K\subset E$.  Moreover, the inequality \eqref{eqn:est2} shows that the energy of its potential is bounded:
$$
\mathcal{E}_1(U_1(1_K \mu), U_1(1_K \mu)) \le \|R_1 \mu\|_\infty \mu(K) < \infty,
$$
which implies $1_K \mu \in \mathcal{S}_0$. 
If $\mu(E) < \infty$, the same logic applies to $\mu$ by taking $K=E$. 
Since $R_1 \mu \in L^\infty$, it follows that $\mu \in \mathcal{S}_{00}$. \hfill \fbox{}

Next we consider the relationship between Radon measures of finite energy integrals and Kato class measures. Lemma \ref{K00isS00} implies that any finite measure $\mu \in \mathcal{K}$ belongs to $\mathcal{S}_{00}$. It is thus natural to ask about the relationship between $\mathcal{K}$ and $\mathcal{S}_{00}$, more precisely, whether the inclusion $\mathcal{S}_{00} \subset \mathcal{K}_{loc}$ holds. By Lemma \ref{S00isalmostK00}, measures in \(\mathcal{S}_{00}\) have a similar property to measures in \(\mathcal{K}\), but we give the following example.

\begin{example}
In general, \(\mathcal{S}_{00} \not \subset \mathcal{K}\). Let \(E:=\mathbb{N} \cup \{0\}\) and 
\[m(dx):=\sum_{n=0}^{\infty}\frac{1}{2^{2n}}\delta_{n}(dx),\]
where \(\delta_n\) is the Dirac delta measure at \(n\). Set \(\lambda(x):=2^x\) and define the jump kernel $Q(x,dy)$ as
 \begin{equation*}
 Q(x,dy):= \left\{ \begin{split} &\ \ \delta_0(dy) \ \ {\rm for\ }x\not =0,\\
  &\sum_{n=1}^\infty \frac{1}{2^n}\delta_n(dy) \ \ {\rm for\ }x=0. \end{split} \right.
 \end{equation*}
Denote by $X$ a pure jump step process with road map $Q$ and speed function $\lambda$. If $X$ starts from $x \in E$, it remains at $x$ until an exponentially distributed random time $T_1$ with rate $\lambda(x)$. At time $T_1$, it jumps to a state $x_1$ according to the distribution $Q(x, dy)$, and remains there for an independent exponentially distributed holding time $T_2$ with rate $\lambda(x_1)$. The process continues in this manner. See \cite[Section 2.2.1]{CF12} for details.

For \(m_0(dx):=\lambda(x)m(dx)\), it holds that \(Q(x,dy)m_0(dx)=Q(y,dx)m_0(dy)\). Hence, by \cite[Section 6.5.(2\(^\circ\))]{CF12}, \(X\) is an \(m\)-symmetric Hunt process and it is a time-changed process of \(X^0\) by \(m\), where \(X^0\) is an \(m_0\)-symmetric pure jump step process whose road map \(Q\) and speed function \(1\). The associated Dirichlet form \((\mathcal{E}, \mathcal{F})\) on \(L^2(E;m)\) is
$$
\left\{
\begin{array}{rl}
\mathcal{E}(f,g) \  := & \dis \frac{1}{2}\iint_{E\times E} (f(x)-f(y))(g(x)-g(y))Q(x,dy)m_0(dy)\\[10pt]
						= & \dis \sum_{n=0}^{\infty}\frac{1}{2^n} (f(n)-f(0))(g(n)-g(0)), \\[15pt]
 \mathcal{F} \  :=   & \dis   \mathcal{F}_e^0 \cap L^2(E;m), 
\end{array}
\right.
$$
where \(\mathcal{F}_e^0\) is an extended Dirichlet space of \(X^0\).

Now, define $\mu(dx):=2^x 1_{\mathbb{N}}(x) m(dx)$ and let $A^\mu$ be the corresponding PCAF. For any $f \in \mathcal{F}$, we have
\begin{align*}
\int_E |f(x)| d\mu(x) &= \sum_{n=1}^{\infty} \frac{1}{2^n} |f(n)|\  \leq \ \sum_{n=1}^{\infty} \frac{1}{2^n} |f(n)-f(0)| +  \sum_{n=1}^{\infty} \frac{1}{2^n} |f(0)|\\
&\leq  \sqrt{\sum_{n=1}^{\infty} \frac{1}{2^n} |f(n)-f(0)|^2} + |f(0)|\ \leq  \sqrt{\mathcal{E}(f,f)} + \|f\|_{L^2(m)}\\
&\leq  \sqrt{2\, \mathcal{E}_1(f,f)}
\end{align*}

showing $\mu \in \mathcal{S}_{0}$. Next, we identify the $1$-potential $U_1\mu$. Set
\begin{equation*}u(x):=
\left\{ 
\begin{split} & \frac{2^x}{1+2^x}(u(0)+1) \quad {\rm for\ }x\not =0, \\[5pt]
& \frac{C_0}{2-C_0} \qquad {\rm for\ }x=0, 
\end{split} 
\right.
\end{equation*}
where $C_0=\sum_{n=1}^{\infty}1/(1+2^n)$. For any $f\in \mathcal{F}$, we have
\begin{eqnarray*}
\mathcal{E}_1(f,u) &=& \sum_{n=1}^{\infty}\frac{1}{2^n} (f(n)-f(0))(u(n)-u(0)) + \sum_{n=0}^{\infty}\frac{1}{2^{2n}} f(n) u(n)\\
&=&\sum_{n=1}^{\infty}\frac{1}{2^n} (f(n)-f(0)) \left(\frac{-1}{1+2^n}u(0)+ \frac{2^n}{1+2^n} \right) +\sum_{n=1}^{\infty}\frac{1}{2^{2n}} f(n) \frac{2^n}{1+2^n}(u(0)+1) + f(0)u(0)\\
&=&\sum_{n=1}^{\infty} f(n) \left( \frac{1}{1+2^n}+\frac{1}{2^n}\frac{1}{1+2^n} \right) + f(0) \sum_{n=1}^{\infty} \left(\frac{1}{2^n} \frac{1}{1+2^n}u(0)- \frac{1}{1+2^n} \right) +f(0)u(0)\\
&=&\sum_{n=1}^{\infty} f(n) \frac{1}{2^n} + f(0) \sum_{n=1}^{\infty} \left(\frac{u(0)}{2^n} - \frac{u(0)}{1+2^n}- \frac{1}{1+2^n} \right) +f(0)u(0)\\
&=&\int_E f(x) \mu(dx) + f(0) u(0) - f(0) u(0) C_0 -f(0)C_0 +f(0)u(0)\\
&=&\int_E f(x) \mu(dx) + f(0) \left(\frac{C_0}{2-C_0}  - \frac{C_0}{2-C_0}  C_0 -C_0 +\frac{C_0}{2-C_0} \right)\\
&=&\int_E f(x) \mu(dx).
\end{eqnarray*}
Thus $U_1\mu=u$, and $U_1\mu$ is bounded by the definition of $u$. Since $\mu(E)=1$, we have $\mu \in \mathcal{S}_{00}.$

Finally, for the first hitting time $\sigma_0$ to $0$ starting from $n \in \mathbb{N}$, we have
\begin{eqnarray*}
\mathbb{E}_n[A_t^\mu] \geq \mathbb{E}_n\left[\int_0^{\sigma_0 \wedge t} 2^n ds\right] = 2^n \int_0 ^t e^{-2^n s} ds = 1-e^{-2^n t}.
\end{eqnarray*}
Therefore, $\sup_{n\in E} \mathbb{E}_n[A_t^\mu] \geq \sup_{n\in \mathbb{N}}(1-e^{-2^n t}) = 1$, which implies $\lim_{t \downarrow 0}\sup_{n\in E} \mathbb{E}_n[A_t^\mu] \geq 1$. Thus, $\mu \not \in \mathcal{K}$.
\end{example}

The previous example demonstrates that the boundedness of the potential is not sufficient to ensure the uniform convergence required for the Kato class, primarily due to the lack of spatial regularity. To bridge this gap, we consider the Feller property, which provides a stronger link between the analytic and probabilistic aspects of the process. The following proposition gives a sufficient condition for a measure in $\mathcal{S}_{00}$ to belong to the Kato class.
 A $C_0$-semigroup $\{S_t\}_{t>0}$ on $C_\infty(E)$  is said to have the Feller property if \(S_t(C_{\infty}(E)) \subset C_{\infty}(E)\) holds for each \(t >0\) and $\|S_tf\|_\infty \le \|f\|_\infty$ for any $f\in C_\infty(E)$, where $C_\infty(E)$ denotes the space of all continuous functions on \(E\) vanishing at infinity.

\begin{proposition}\label{FellerS00isK}
Suppose that the semigroup \(\{P_t\}_{t>0}\) is a Feller semigroup and \(U_1\mu \in C_\infty(E)\) for \(\mu \in \mathcal{S}_{00}\). Then \(\mu \in \mathcal{K}_{00}\).
\end{proposition}

\begin{proof}
We first remark that for \(\mu \in \mathcal{S}_{00}\), \(R_1\mu\) is a quasi-continuous version of the \(1\)-potential \(U_1\mu\). By the Markov property, we have for q.e. \(x \in E\):
\begin{eqnarray*}
U_1\mu(x) &=& \mathbb{E}_x\left[ \int_0^\infty e^{-s} dA_s^\mu \right] = \mathbb{E}_x\left[ \int_0^t e^{-s} dA_s^\mu \right] + \mathbb{E}_x\left[ \int_t^\infty e^{-s} dA_s^\mu \right] \\
&=& \mathbb{E}_x\left[ \int_0^t e^{-s} dA_s^\mu \right] + e^{-t} \mathbb{E}_x\left[ U_1\mu(X_t) \right].
\end{eqnarray*}
Since \(U_1\mu \in C_\infty(E)\) by assumption, the quasi-continuous version coincides with the continuous one, and the above identity holds for all \(x \in E\). Combining this with the strong continuity of \(\{P_t\}\) on \(C_\infty(E)\), we obtain
\begin{eqnarray*}
\sup_{x \in E} \mathbb{E}_x\left[ \int_0^t e^{-s} dA_s^\mu \right] &=& \sup_{x \in E} \left| U_1\mu(x) - e^{-t} P_t U_1\mu(x) \right| \\
&\leq& \| U_1\mu - P_t U_1\mu \|_\infty + (1 - e^{-t}) \| U_1\mu \|_\infty \xrightarrow[t \to 0]{} 0.
\end{eqnarray*}
It follows that
\[ \lim_{t \to 0} \sup_{x \in E} \mathbb{E}_x [A_t^\mu] \leq \lim_{t \to 0} e^t \sup_{x \in E} \mathbb{E}_x \left[ \int_0^t e^{-s} dA_s^\mu \right] = 0. \]
Thus, \(\mu \in \mathcal{K}\). Since \(\mu(E) < \infty\) is already given by \(\mu \in \mathcal{S}_{00}\), we conclude \(\mu \in \mathcal{K}_{00}\).
\end{proof}

Next, we assume that the Dirichlet form $(\form, \dom)$ is transient.  Then the 0-order capacity ${\rm Cap}_{(0)}$ is well-defined  as a Choquet capacity (see \cite[Section 2.2]{FOT11}).  Let $R\mu(x) := \mathbb{E}_x[A_{\zeta}^\mu]$  for a PCAF \(A\) associated with \(\mu\). Note that, for \(\mu \in \mathcal{S}_0^{(0)}\), \(R\mu\) is a quasi-continuous version of \(U\mu:=U_0\mu\). Then the corresponding results for the transient case are obtained as follows:

\begin{proposition} Assume that the Dirichlet form $(\form, \dom)$ is transient and $\mu\in {\cal S}$ and its potential $R\mu$ belongs to $L^\infty(E;m)$. 
Let $\{F_\ell\}$ be a (compact) nest attached to $\mu$, that is, $1_{F_\ell}\mu \in {\cal S}^{(0)}_0$ holds for each $\ell \in{\mathbb N}$.
Then for all compact sets $K\subset E$ and $\ell \in{\mathbb N}$, the following inequalities hold$:$
$$
\mu(K\cap F_\ell)  \le  \sqrt{{\rm Cap}_{(0)}(K)} \sqrt{\form(U(1_{K\cap F_\ell}\mu),U(1_{K\cap F_\ell}\mu))} 
$$
and
$$
\form(U(1_{K\cap F_\ell}\mu),U(1_{K\cap F_\ell}\mu)) \le \|R\mu\|_\infty \mu(K \cap F_\ell).
$$
In particular, we have$:$
$$
\mu(K) \le {\rm Cap}_{(0)}(K) \|R\mu\|_{\infty}<\infty.
$$
This implies that $\mu$ is a Radon measure.
\end{proposition}

\noindent
{\it Proof}. The proof is analogous to those of Lemma \ref{Lem:smooth-01} and Proposition \ref{Prop:smooth}. The essential difference lies in utilizing the characterization of the $0$-order potential $U\nu$ for a measure $\nu \in \mathcal{S}_0^{(0)}$, which satisfies:
$$
\int_E \tilde{\varphi}(x) \nu(dx) = \mathcal{E}(U\nu, \varphi) \quad \text{for any } \varphi \in \mathcal{F}_e.
$$
By applying this identity with $\nu = 1_{K \cap F_\ell}\mu$ and using the properties of the $0$-order capacity, the desired inequalities follow in the same manner as in the proof of Proposition \ref{Prop:smooth}.
\hfill \fbox{}

\begin{corollary} Assume that the Dirichlet form $(\form, \dom)$ is transient. The following hold$:$
\begin{itemize}
\item[\sf (1)]
If a smooth measure $\mu$ satisfies $R\mu \in L^\infty(E;m)$ and $\mu(E)<\infty$, then $\mu$ belongs to ${\cal S}_0^{(0)}$.

\item[\sf (2)] {\rm({\it cf.} \cite[Theorem 2.4.2]{FOT11})}
 If $\mu$ is a smooth measure such that $R\mu \in L^\infty(E;m)$, then
$$
\int_E \tilde{u}(x)^2 \mu(dx) \le \| R\mu\|_\infty \form(u,u)<\infty, \quad u \in \dom_e.
$$
In particular, this implies that the embedding $\tilde{\dom}_e \subset L^2(E;\mu)$ is continuous, where $\tilde{\dom}_e$ is the set of all quasi-continuous modifications of functions 
in $\dom_e$.

\end{itemize}
\end{corollary}
 
The second statement in the corollary can be deduced directly as the transient case of the Stollmann-Voigt inequality (Proposition \ref{SVineq}).

We close this section by presenting a key example of smooth measures associated with Brownian motion.

\begin{example}
  \label{ex:BM}
Let $\mathbb{M}$ be Brownian motion on $\mathbb{R}^d$ ($d \ge 3$) and 
$R_1(x,y)$ be its $1$-resolvent kernel, that is,
\begin{align}
\label{eq:1}
  R_{1}(x,y) &:= \int_{0}^{\infty} e^{-t} \frac{1}{(2\pi t)^{d/2}} 
               \exp \left\{ - \frac{|x-y|^{2}}{2t} \right\} dt.
\end{align}

Define a measure $\mu$ on $\mathbb{R}^d$ by $\mu(dx) = |x|^{-\beta} dx$ for $\beta \in \mathbb{R}$. It is known in \cite[Example 5.1.1]{FOT11} that $\mu$ is smooth
for any $\beta \in \mathbb{R}$. 
We investigate the condition under which the measure $\mu$ belongs to
the class $\mathcal{S}_0$.
Since $\mu \in S_0$ requires $\mu$ to be a Radon measure, we must have $\beta < d$, which will be assumed hereafter. 
By \cite[Exercise 4.2.2]{FOT11}, $\mu \in \mathcal{S}_0$ if and only if its mutual energy integral is finite:
\begin{equation}
    I = \iint_{\mathbb{R}^d \times \mathbb{R}^d} R_1(x, y) \mu(dx) \mu(dy) 
      = \iint_{\mathbb{R}^d \times \mathbb{R}^d} R_1(x, y) |x|^{-\beta} |y|^{-\beta} 
        \,dx\, dy < \infty.
\end{equation}
Since $R_{1}(x,y)$ is represented by the modified Bessel function of the second kind of order $d/2 - 1$, the asymptotic behavior of $R_{1}(x,y)$ is given by:
\begin{align}
    R_1(x, y) &\asymp |x-y|^{2-d} \quad \text{as } |x-y| \to 0, 
    \label{eq:kernel_s} \\
    R_1(x, y) &\asymp |x-y|^{(1-d)/2}e^{-|x-y|} \quad \text{as } |x-y| \to \infty.
    \label{eq:kernel_l}
\end{align}
To evaluate the convergence of $I$, we decompose $\mathbb{R}^d \times \mathbb{R}^d$ into four mutually disjoint regions based on the distance and the position from the origin:
\begin{align}
\label{eq:decomp}
  D_{1} &= \{(x,y) : |x-y| \leq 1,\ |x| \leq 1,\ |y| \leq 1\}, \\
  D_{2} &= \{(x,y) : |x-y| \leq 1,\ |x| > 1,\ |y| > 1\}, \\
  D_{3} &= \{(x,y) : |x-y| > 1\}, \\
  D_{4} &= \{(x,y) : |x-y| \leq 1\} \setminus (D_1 \cup D_2).
\end{align}
Note that $D_4$ is a mixed region where exactly one of $|x|, |y|$ is at most $1$ and the other exceeds $1$. On the set $D_4$, both $|x|$ and $|y|$ are of order $O(1)$, so $|x|^{-\beta}|y|^{-\beta}$ has at most an integrable singularity. Thus, the contribution from $D_4$ is finite for $\beta < d$ and does not affect the critical convergence conditions.
We write
\begin{align}
\label{eq:Isum}
  I = J_{1} + J_{2} + J_{3} + J_{4},
  \quad J_k := \iint_{D_k} R_1(x,y)\,|x|^{-\beta}|y|^{-\beta}\,dx\,dy \quad (k = 1,2,3,4).
\end{align}
Since $J_4 < \infty$ for $\beta < d$, it suffices to focus on $J_{1}, J_{2}$, and $J_{3}$.

\begin{enumerate}
\item \text{Estimate of $J_1$ (short-range, near the origin).} \\
For $(x,y) \in D_1$, we have $|x-y| \leq 1$. By the short-range behavior of the kernel \eqref{eq:kernel_s}, we have $R_1(x,y) \asymp |x-y|^{2-d}$.
Let $A_x = \{y : |y| \leq 1,\ |x-y| \leq 1\}$. We partition the domain $A_x$ into three mutually disjoint subregions:
\begin{align}
\label{eq:Bdef}
  B_1 &= \left\{ y \in A_x : |y| < |x|/2 \right\}, \\
  B_2 &= \left\{ y \in A_x : |x-y| < |x|/2 \right\}, \\
  B_3 &= \left\{ y \in A_x : |y| \geq |x|/2,\ |x-y| \geq |x|/2 \right\}.
\end{align}
We write $J_1 \asymp J_{1,1} + J_{1,2} + J_{1,3}$, where each term represents the contribution from $B_i$:
\begin{align}
  J_{1,i} := \int_{|x| \le 1} |x|^{-\beta} \left( \int_{B_i} |x-y|^{2-d} |y|^{-\beta} \,dy \right) dx \quad (i=1,2,3).
\end{align}

First, for $y \in B_1$, the reverse triangle inequality gives $|x-y| \asymp |x|$. Assuming the local integrability condition $\beta < d$, the inner integral behaves as:
\begin{align}
    \int_{B_1} |x-y|^{2-d} |y|^{-\beta} \,dy 
    \asymp |x|^{2-d} \int_{|y| < |x|/2} |y|^{-\beta} \,dy 
    \asymp |x|^{2-\beta}.
\end{align}
Integrating this over $x$ yields $J_{1,1} \asymp \int_{|x| \le 1} |x|^{-\beta} \cdot |x|^{2-\beta} \,dx = \int_{|x| \le 1} |x|^{2-2\beta} \,dx$.

Second, for $y \in B_2$, the condition $|x-y| < |x|/2$ implies $|y| \asymp |x|$. By changing variables to $z = y-x$:
\begin{align}
    \int_{B_2} |x-y|^{2-d} |y|^{-\beta} \,dy 
    \asymp |x|^{-\beta} \int_{|z| < |x|/2} |z|^{2-d} \,dz 
    \asymp |x|^{2-\beta}.
\end{align}
This similarly gives $J_{1,2} \asymp \int_{|x| \le 1} |x|^{-\beta} \cdot |x|^{2-\beta} \,dx = \int_{|x| \le 1} |x|^{2-2\beta} \,dx$.

Third, for $y \in B_3$, the triangle inequality ensures $|x-y| \asymp |y|$, so the kernel behaves as $|y|^{2-d}$. 
Since the integrands are non-negative, we apply Tonelli's theorem to exchange the order of integration. For a fixed $y$, let $B_3(y) = \{ x : (x,y) \in B_3 \}$ be the $x$-section of the integration domain. We rigorously evaluate the inner integral $\int_{B_3(y)} |x|^{-\beta} \,dx$ using strict integral bounds.

For the upper bound, the condition $|y| \ge |x|/2$ combined with the overarching domain restriction $x \in D_1$ ($|x| \le 1$) implies $B_3(y) \subset \{x : |x| \le \min(1, 2|y|)\}$. Since the integrand is strictly positive, we can safely enlarge the integration domain to the full ball $\{x : |x| \le 2|y|\}$ to obtain a valid upper bound:
\begin{align}
    \int_{B_3(y)} |x|^{-\beta} \,dx 
    \le \int_{|x| \le 2|y|} |x|^{-\beta} \,dx 
    = \omega_d \int_0^{2|y|} r^{d-1-\beta} \,dr 
    = \frac{\omega_d \, 2^{d-\beta}}{d-\beta} |y|^{d-\beta}.
\end{align}

For the lower bound, we restrict the domain of integration to $\{x : |x| \le |y|/2\}$. If we assume $|y| \le 1/2$, any $x$ in this restricted domain satisfies $|x-y| \ge |y| - |x| \ge |y|/2 \ge |x|/2$. Moreover, $|x-y| \le |y| + |x| \le 3|y|/2 \le 3/4 < 1$. Thus, $\{x : |x| \le |y|/2\} \subset B_3(y)$ holds for $|y| \le 1/2$, giving:
\begin{align}
    \int_{B_3(y)} |x|^{-\beta} \,dx 
    \ge \int_{|x| \le |y|/2} |x|^{-\beta} \,dx 
    = \omega_d \int_0^{|y|/2} r^{d-1-\beta} \,dr 
    = \frac{\omega_d \, (1/2)^{d-\beta}}{d-\beta} |y|^{d-\beta}.
\end{align}

These exact integral evaluations demonstrate that, for $|y| \le 1/2$, the inner integral is bounded both above and below by constant multiples of $|y|^{d-\beta}$. (The contribution from the annulus $1/2 < |y| \le 1$ is finite and does not affect the singularity). Consequently, we have the exact asymptotic equivalence:
\begin{align}
    \int_{B_3(y)} |x|^{-\beta} \,dx \asymp |y|^{d-\beta}.
\end{align}

Substituting this precise integral evaluation back into the outer integral yields:
\begin{align}
    J_{1,3} &\asymp \int_{|y| \le 1} |y|^{2-d-\beta} \left( \int_{B_3(y)} |x|^{-\beta} \,dx \right) dy \notag \\
            &\asymp \int_{|y| \le 1} |y|^{2-d-\beta} \cdot |y|^{d-\beta} \,dy \notag \\
            &= \int_{|y| \le 1} |y|^{2-2\beta} \,dy.
\end{align}

Combining all three components, the total integral $J_1$ behaves exactly as:
\begin{align}
\label{eq:J1final}
  J_{1} \asymp J_{1,1} + J_{1,2} + J_{1,3} \asymp \int_{|x| \leq 1} |x|^{2-2\beta} \,dx 
        = \omega_d \int_0^1 r^{d+1-2\beta} \,dr.
\end{align}
This integral converges if and only if the exponent $d+1-2\beta > -1$, which yields exactly the condition:
\begin{align}
  J_1 < \infty \iff \beta < \frac{d+2}{2}.
\end{align}
\item \text{Estimate of $J_2$ (short-range, away from the origin).} \\
Since $(x,y) \in D_2$ implies $|x| > 1$, $|y| > 1$, and $|x-y| \leq 1$, we again have $R_1(x,y) \asymp |x-y|^{2-d}$. 
Furthermore, $|x-y| \leq 1$ and $|x| > 1$ imply $|y| \geq |x| - 1 > |x|/2$, so $|y| \asymp |x|$. Therefore, the inner integral evaluates to:
\begin{align}
  \int_{\substack{|y|>1 \\ |y-x| \leq 1}} |x-y|^{2-d}|y|^{-\beta}\,dy
  \asymp |x|^{-\beta} \int_{|z| \leq 1}|z|^{2-d}\,dz
  \asymp |x|^{-\beta},
\end{align}
which leads to
\begin{align}
\label{eq:J2final}
  J_2 \asymp \int_{|x|>1} |x|^{-\beta} \cdot |x|^{-\beta}\, dx
           = \omega_d \int_1^{\infty} r^{d-1-2\beta}\, dr.
\end{align}
This converges if and only if the exponent $d-1-2\beta < -1$:
\begin{align}
  J_2 < \infty \iff \beta > \frac{d}{2}.
\end{align}
\item \text{Estimate of $J_3$ (long-range).} \\
For $(x,y) \in D_3$, we have $|x-y| > 1$, where the kernel exhibits exponential decay. To evaluate the convergence, we can write $R_1(x,y) \asymp e^{-c|x-y|}$ for some effective constant $c > 0$. Setting $z = y - x$, the integral becomes:
\begin{equation}
    J_3 \asymp \int_{\mathbb{R}^d}|x|^{-\beta} \left(\int_{|z|>1} e^{-c|z|}|x+z|^{-\beta}\,dz\right)dx =: \int_{\mathbb{R}^d}|x|^{-\beta} I(x)\,dx.
\end{equation}
We evaluate the inner integral $I(x)$ by dividing the domain of $x$ into a compact region $\{|x| \le 2\}$ and a long-range region $\{|x| > 2\}$.

For $|x| \le 2$, the singularity at $z = -x$ is integrable since $\mu \in \mathcal{S}_0$ implies $\beta < d$, and the exponential factor ensures convergence at infinity. Thus, $I(x)$ is a strictly positive, continuous function of $x$ on the compact set $\{|x| \le 2\}$, which implies $I(x) \asymp 1$. The contribution to $J_3$ from this region is finite:
\begin{align}
    \int_{|x|\le 2} |x|^{-\beta} I(x) \,dx \asymp \int_{|x|\le 2} |x|^{-\beta} \,dx < \infty \quad (\text{for } \beta < d).
\end{align}

For $|x| > 2$, we split the $z$-integration domain into a main body $\{1 < |z| \le |x|/2\}$ and a tail $\{|z| > |x|/2\}$.
On the main body $\{1 < |z| \le |x|/2\}$, the triangle inequalities yield $|x|/2 \le |x+z| \le 3|x|/2$, guaranteeing that $|x+z|^{-\beta} \asymp |x|^{-\beta}$. Thus, this region provides both upper and lower bounds:
\begin{align}
    \int_{1 < |z| \le |x|/2} e^{-c|z|} |x+z|^{-\beta} \,dz 
    \asymp |x|^{-\beta} \int_{1 < |z| \le |x|/2} e^{-c|z|} \,dz 
    \asymp |x|^{-\beta}.
\end{align}
(The last equivalence holds because the integral of $e^{-c|z|}$ over $\{1 < |z| \le |x|/2\}$ is strictly bounded away from zero and converges to a finite constant as $|x| \to \infty$).

On the tail region $\{|z| > |x|/2\}$, we extract an upper bound by splitting the exponential factor $e^{-c|z|} \le e^{-c|x|/4} e^{-c|z|/2}$:
\begin{align}
    \int_{|z|>|x|/2} e^{-c|z|} |x+z|^{-\beta} \,dz 
    \le e^{-c|x|/4} \int_{\mathbb{R}^d} e^{-c|z|/2} |x+z|^{-\beta} \,dz 
    \lesssim e^{-c|x|/4}.
\end{align}
Since the exponential decay $e^{-c|x|/4}$ approaches zero faster than any polynomial, it is $o(|x|^{-\beta})$ and gets absorbed into the main term.

Therefore, for all $|x| > 2$, we establish the exact asymptotic behavior $I(x) \asymp |x|^{-\beta}$. The long-range contribution to $J_3$ behaves as:
\begin{equation}
    \int_{|x|>2} |x|^{-\beta} I(x)\,dx \asymp \int_{|x|>2} |x|^{-2\beta}\,dx \asymp \int_2^\infty r^{d-1-2\beta}\,dr.
\end{equation}
This integral converges if and only if $d-1-2\beta < -1$. Hence, we obtain the condition:
\begin{align}
  J_3 < \infty \iff \beta > \frac{d}{2}.
\end{align}
\end{enumerate}

Combining the estimates for $J_1$, $J_2$, and $J_3$, we obtain the following convergence conditions:
\begin{align}
  \begin{cases}
    J_1 < \infty &\iff \beta < \dfrac{d+2}{2}, \medskip \\
    J_2 < \infty &\iff \beta > \dfrac{d}{2}, \medskip \\
    J_3 < \infty &\iff \beta > \dfrac{d}{2}.
  \end{cases}
\end{align}
The binding constraints are exactly given by these inequalities. Therefore, the singular measure $\mu(dx) = |x|^{-\beta}dx$ belongs to the finite energy class $\mathcal{S}_0$ if and only if
\begin{align}
\label{eq:S0}
 \frac{d}{2} < \beta < \frac{d+2}{2}.
\end{align}

Next let us consider a condition on $\beta$ for the Kato class.
For dimensions $d \ge 3$, a measure $\mu$ belongs to the Kato class $\mathcal{K}$ 
if and only if it satisfies the local uniform integrability condition with respect to the Newtonian potential:
\begin{align}
\label{eq:kato_def}
  \lim_{a \to 0} \sup_{x \in \mathbb{R}^d} \int_{|x-y| \le a} |x-y|^{2-d} \mu(dy) = 0.
\end{align}
Put
\begin{align}
\label{eq:4}
  I(x,a) &:= \int_{|x-y| \le a}^{ } |x-y|^{2-d} |y|^{-\beta} dy. 
\end{align}
We divide the argument into several parts.
\begin{enumerate}
\item $\beta < 0$: Let $\alpha = -\beta > 0$,
  so that the measure is $\mu(dy) = |y|^\alpha \,dy$.
  For any $y \in B(x, a)$, the reverse triangle inequality gives $|y| \ge |x| - |x-y| \ge |x| - a$.
  Assuming $|x| > a$, we can bound the integral from below:
  \begin{align}
    I(x,a) 
    &= \int_{|x-y| \le a} |x-y|^{2-d} |y|^{\alpha} \,dy 
      \ge \int_{|x-y| \le a} |x-y|^{2-d} (|x| - a)^{\alpha} \,dy \\
    &= (|x| - a)^\alpha \omega_d \int_0^{a} r^{2-d} r^{d-1} \,dr
      = \frac{\omega_d \, a^2}{2} (|x| - a)^\alpha.
  \end{align}
  Since $\alpha > 0$, taking the supremum over all $x \in \mathbb{R}^d$ forces $|x| \to \infty$, which yields:
  \begin{align}
    \sup_{x \in \mathbb{R}^d} \int_{|x-y| \le a} |x-y|^{2-d} |y|^\alpha \,dy
    \ge \lim_{|x| \to \infty} \frac{\omega_d \, a^2}{2} (|x| - a)^\alpha = \infty.
  \end{align}
  This implies that $\mu \not\in \mathcal{K}$.
\item $\beta = 0$: It is clear that $\mu$ belongs to $\mathcal{K}$.
\item $0 < \beta < 2$: By Riesz' rearrangement inequality, we obtain that for some constant $C > 0$, 
  \begin{align}
    \label{eq:3}
    \sup_{x \in \mathbb{R}^{d}} I(x,a)
    &= \sup_{x \in \mathbb{R}^{d}} \int_{|x-y| \le a}^{ } |x-y|^{2-d} |y|^{-\beta} dy \\
    &\le \sup_{x \in \mathbb{R}^{d}} \int_{|y| \le a}^{ } |y|^{2-d-\beta} dy
      \le C a^{2-\beta} \to 0 \quad \text{as}\ a \to 0.
  \end{align}
  Hence it holds that $\mu \in \mathcal{K}$.
\item $\beta \ge 2$: we see
\begin{align}
\label{eq:5}
  \sup_{x \in \mathbb{R}^{d}} I(x,a)
  &\ge I(0,a) = \int_{|y| \le a}^{ } |y|^{2-d} |y|^{-\beta} dy = \omega_{d} \int_0^{ a} r^{1-\beta} dr = \infty, 
\end{align}
so $\mu \not\in \mathcal{K}$. 
\end{enumerate}
In conclusion, we have
\begin{align}
\label{eq:Kato}
  \mu \in \mathcal{K} \iff 0 \le \beta < 2.
\end{align}

Combining \eqref{eq:S0} and \eqref{eq:Kato},
we arrive at the following conclusion:
for $d=3$ with $3/2 < \beta < 2$, we have $\mu \in \mathcal{K} \cap \mathcal{S}_0$,
whereas for $d \ge 4$, the intersection becomes empty
($\mathcal{K} \cap \mathcal{S}_0 = \emptyset$).
It is instructive to compare this result with Lemma \ref{K00isS00}.
\end{example}

\section{Topologies on the Dynkin class and the Kato class}\label{sec_Miyadera}
In this section, we consider a suitable metric on the Dynkin class  \(\mathcal{D}\) and Kato class \(\mathcal{K}\). Since a Kato class measure \(\mu\in \mathcal{K}\) is characterized by \(\lim_{\alpha \to \infty}\|R_\alpha \mu\|_\infty =0\), one might think \(\|R_1\mu-R_1\nu\|_{\infty}\) for \(\mu, \nu \in \mathcal{K}\) is an appropriate metric, but as the following example shows, this metric is not complete.

\begin{example}\label{counterexofmetric1}
We show the metric \(d(\mu, \nu):=\|R_1\mu-R_1\nu\|_{\infty}\) on \(\mathcal{K}\) is not complete in general. Let \(E\) be the unit disk \(B(0;1)\) of \(\mathbb{R}^2\), \(m\) be the Lebesgue measure and \(X\) be an absorbed Brownian motion on \(E\). Denote by \(\tau\) the first exit time from \(E\), and \(\tau_n\) the first exit time from \(B(0;1-\frac{1}{n})\), the open ball centred at \(0\) with radius \(1-\frac{1}{n}\). We remark that  \(\tau = \zeta\), where \(\zeta\) is the lifetime of \(X\). Define \(u_n:=\mathbb{E}_x[e^{-\tau_n}]\) and \(u:=\mathbb{E}_x[e^{-\tau}]\). Then, by \cite[Lemma 2.3.10, Corollary 3.2.3]{CF12}, there exists a measure \(\mu_n\) called a \(1\)-order equilibrium measure of \(B(0;1-\frac{1}{n})^c\) such that \(\mu_n \in \mathcal{S}_0\) and \(u_n\) is a quasi-continuous version of \(R_1\mu_n\). Hence \(u_n(x)=R_1\mu_n(x)\). Since \(u_n\) is bounded and \(\mu_n(E)={\rm Cap}_1(B(0;1-1/n)^c)<\infty\), it holds that \(\mu_n \in \mathcal{S}_{00}.\) The absorbing Brownian motion \(X\) is a Feller process (\cite{C86}) and \(u_n\) is continuous (\cite[Theorem 2.1]{KKT17}), by Proposition \ref{FellerS00isK}, \(\mu_n\) is a Kato class measure.

Denote \(I_0\) by the modified Bessel function of order \(0\),
\[I_0(x):=\sum_{k=0}^\infty \frac{(\frac{x}{2})^{2k}}{(k!)^2}.\] Since \(u_n(x)\) (resp. \(u(x)\)) is a Laplace transform of the first hitting time of \(1-1/n\) (resp. 1) for the  $2$-dimensional Bessel process  (equivalently, Bessel process of order $0$) starting at \(|x|\), we have
\[u_n(x)= \begin{cases}\frac{I_0(\sqrt{2}|x|)}{I_0(\sqrt{2} (1-\frac{1}{n}))} &\text{\ for\ }|x|\leq 1-\frac{1}{n},\\
\hspace{7mm}1 &\text{\ for\ } 1-\frac{1}{n}<|x|<1, \end{cases} \]
and
\[u(x)=\frac{I_0(\sqrt{2}|x|)}{I_0(\sqrt{2})}, \quad |x| \le 1.\]
For example, see \cite[p.513]{BS02} for these formulae. By the continuity and the monotonicity of \(I_0\), \(u_n=R_1\mu_n\) converges to \(u\) uniformly on \(E\).

However, \(u\) is not a potential of any smooth measure. Indeed, if there exists a smooth measure \(\mu\) with corresponding PCAF \(A\) such that \(u=R_1\mu\), then the following identity must hold:
\begin{eqnarray*}
u(x)&=&\mathbb{E}_x[e^{-\tau}\mathbb{E}_{x}\left[e^{-\tau}\circ \theta_{\tau}|\mathcal{F}_{\tau}]\right]\\
&=&\mathbb{E}_x[e^{-\tau}u(X_{\tau})]\\
&=& \mathbb{E}_x\left[e^{-\tau}\mathbb{E}_{X_{\tau}}\left[\int_0^\zeta e^{-t}dA_t\right]\right]\\
&=&0,
\end{eqnarray*}
which is a contradiction.
\end{example}

\begin{remark}
\begin{enumerate}
\item In Example \ref{counterexofmetric1}, since the topological support of \(\mu_n\) is \(B(0;1-1/n)^c\), \(\mu_n\) converges vaguely to \(0\). Hence the metric \(d_{\mathcal{K}}(\mu, \nu):= \|R_1 \mu-R_1\nu\|_q + d_{C_c^*}(\mu, \nu)\) on \(\mathcal{K}\) is also not complete, where \(d_{C_c^*}\) is a metric inducing the vague topology. 
\item Noda \cite{N26} established convergence results for PCAFs associated
with smooth measures in the local Kato class under the uniform local
Kato-type condition. His framework allows underlying spaces, processes, and heat kernels to vary, providing a flexible setting for specific convergence problems. However, the primary focus is on establishing convergence in concrete situations rather than on introducing a topology on the space of measures.
\end{enumerate}
\end{remark}

The following topology is inspired by the Miyadera norm in \cite{OSSV96}.
\begin{definition}
For \(\mu, \nu \in \mathcal{D}\), we set
{
$$ 
d_\mathcal{D}(\mu, \nu) := \big\| R_1 |\mu-\nu|  \big\|_\infty := \big\| R_1 (\mu-\nu)_+ + R_1 (\mu-\nu)_- \big\|_\infty, 
$$} 
where \((\mu-\nu)_+\) and \((\mu-\nu)_-\) are the positive and negative part of \(\mu-\nu\) obtained by the Jordan decomposition theorem. 
We call \(d_\mathcal{D}\) the {\it Miyadera metric}.
\end{definition}

We first verify the well-definedness of the metric $d_D$. By the Jordan decomposition theorem, the positive part \((\mu-\nu)_+\) and the negative part \((\mu-\nu)_-\) of \(\mu-\nu\) are unique non-negative measures possessing a measurable set \(E_+\) such that \((\mu-\nu)_+(E_+^c)= (\mu-\nu)_-(E_+) = 0\). Hence, for a nest \(\{F_k\}_k\) satisfying \(\mu(F_k)<\infty\) and \(\nu(F_k)<\infty\), it holds that \((\mu-\nu)_+(F_k) = (\mu-\nu)(F_k \cap E_+) <\mu(F_k)+\nu(F_k) <\infty \), and for a set \(A\) with \(\Capa(A)=0\), we have \((\mu-\nu)_+(A)= (\mu-\nu)(A \cap E_+) \leq \mu(A)+\nu(A) =0\). The same statements holds for \((\mu-\nu)_-\) and so \((\mu-\nu)_+\) and \((\mu-\nu)_-\) are smooth measures.

\begin{remark}
According to \cite{OSSV96}, the Miyadera norm $\|B\|_P$ of an operator $B$ on a Banach space $(\mathcal{B}, \|\cdot\|_{\mathcal{B}})$ is defined as
\[\|B\|_P := \sup_{f \in \mathcal{D}(L), \|f\|_\mathcal{B}\leq 1} \int_0^1 \|BP_tf\|_\mathcal{B} \,dt\]
where \(P=\|P_t\|_t\) is a strongly continuous semigroup on \(\mathcal{B}\) and \(L\) is its associated generator. For our situation, we set \(\mathcal{B}:=L^1(E;m)\) and \(B(f):=Vf\) for a function \(V\), and it holds that
\[\|V\|_P = \sup_{f \in \mathcal{D}(L), \|f\|_{L^1}\leq 1} \int_E \int_0^1 P_t|V|\,dt\,|f|\,dm = \sup_x \int_0^1 P_t|V|(x)\,dt\]
by the symmetry for \(P_t\) and the \(L^1\)-\(L^\infty\) duality. Similarly to the proof of Proposition \ref{FellerS00isK}, we have \((1-e^{-1})\|R_1|V|\|_\infty \leq \|\int_0^1 P_t|V|\,dt\|_\infty\). Combining this with \(\int_0^1 P_t|V|dt \leq e \int_0^1 e^{-t}P_t|V|dt \leq eR_1|V|\), \(\|R_1|V|\|_\infty\) is comparable to the Miyadera norm \(\|V\|_P \) for a function \(V\). Hence, the set of all functions with \(\|V\|_P<\infty\) coincides with the set of Dynkin class measures absolutely continuous with respect to the underlying measure. Moreover, the set of all operators with bounded Miyadera norm is complete, and the family corresponding to the set of all Kato class measures (absolutely continuous with respect to $m$) is closed under the Miyadera norm \cite[Proposition 3.1]{V95}. Since measures cannot be approximated by functions with respect to the Miyadera metric, the Miyadera norm for measures is not discussed in \cite{OSSV96}; the authors instead employed a different method to study Feller properties for measure-perturbed generators.
\end{remark}

Before considering the completeness with respect to \(d_\mathcal{D}\), we see the following lemma.
\begin{lemma}\label{pot_lem}
Let $\mu_n, \mu$ be Radon measures on $E$. If $\mu_n$ converges vaguely to $\mu$, then the inequality
\[\int_E f \,d\mu \leq \varliminf_n \int_E f \,d\mu_n\]
        holds for any non-negative lower semi-continuous function $f$ on $E$.
\end{lemma}
\begin{proof}
By \cite[Proposition 7.1.1.e]{F99}, for a non-negative lower semi-continuous function \(f\), it holds that
\[\int_E f \,d\mu = \sup{\left\{\int_E \varphi\,d\mu : \varphi \in C_c, 0\leq \varphi \leq f\right\}}.\]
For \(\varphi \in C_c\) satisfying \(0\leq \varphi \leq f\), we have
\begin{eqnarray*}
\int_E \varphi \, d\mu = \varliminf_n \int_E \varphi \, d\mu_n \leq \varliminf_n \int_E f \, d\mu_n
\end{eqnarray*}
and by taking the supremum, the proof is completed.
\end{proof}

\begin{theorem}\label{Dynkin_complete}
Suppose that \(R_1f\) has a lower semi-continuous version for any non-negative function \(f\in C_c.\) Then \((\mathcal{D}, d_{\mathcal{D}})\) is a complete metric space.
\end{theorem}

\begin{proof}
It is easy to see that \(d_{\mathcal{D}}\) is a pseudo metric on \(\mathcal{D}\). If \(R_1|\mu-\nu|=0\), then \(A^{(\mu-\nu)_+}\) and \(A^{(\mu-\nu)_-}\) are \(0\) and so are \(A^{\mu-\nu}\) and \(\mu-\nu\). Hence \((\mathcal{D}, d_{\mathcal{D}})\) is a metric space.

For any compact set \(K\), we take a non-negative function \(\varphi \in \mathcal{F}\cap C_c\) satisfying \(\varphi =1\) on \(K\). Then, by Stollmann--Voigt's inequality (Proposition \ref{SVineq}), we have
\begin{eqnarray}
|\mu-\nu|(K) \leq \int_E \varphi^2 \,d|\mu-\nu| \leq \|R_1|\mu-\nu|\|_\infty\,\mathcal{E}(\varphi, \varphi) \label{eq:Dynkin_complete_1}
\end{eqnarray}
for any \(\mu, \nu \in \mathcal{D}\).

We take a \(d_{\mathcal{D}}\)-Cauchy sequence \(\{\mu_n\}_n \subset \mathcal{D}\) and \(f\in C_c\). By the regularity of a Dirichlet form, for any \(\varepsilon>0\), we can take \(f_\varepsilon \in \mathcal{F}\cap C_c\) such that \(\|f-f_\varepsilon\|_\infty \leq \varepsilon .\) We set \(K:=\) supp\(f\) and by the Stollmann-Voigt inequality and (\ref{eq:Dynkin_complete_1}), we have
\begin{eqnarray*}
\left|\int_E f\, d\mu_n - \int_E f\, d\mu_m \right| &\leq & \int_E |f|\, d|\mu_n - \mu_m|\\
& \leq & \sqrt{\int_E |f|^2\, d|\mu_n - \mu_m|}\,\sqrt{|\mu_n - \mu_m|(K)}\\
&\leq & \big\|R_1 |\mu_n-\mu_m| \big\|_\infty \, \sqrt{\mathcal{E}_1(f,f)}\, \sqrt{\mathcal{E}_1(\varphi, \varphi)} .
\end{eqnarray*}
Hence \(\{\mu_n\}_n\) is a Cauchy sequence with respect to the vague topology and by \cite[Theorem A 2.3]{K02}, there exists a Radon measure \(\mu\) such that \(\mu_n\) converges vaguely to \(\mu\). By (\ref{eq:Dynkin_complete_1}), \(|\mu_n-\mu|\) also converges to \(0\) vaguely. For \(M:=\sup_n R_1\mu_n < \infty\) and the \(1\)-equilibrium potential \(e_A\) of a relatively compact open set \(A\), it holds that
\[\mu(A)\leq \varliminf_n \mu_n(A) \leq \varliminf_n \int |e_A|^2 d\mu_n \leq \mathcal{E}_1(e_A,e_A)\, M^2 =\Capa(A)\, M^2.\]
Since \(E\) is a locally compact space, for any open set \(U\), there exists a sequence of relatively compact open sets \(\{A_n\}_n\) such that \(A_n \nearrow U\). Since \(\mu\) is a Radon measure, \(\mu(A_n)\) converges to \(\mu(U)\), so \(\mu(U)\leq \Capa(U)\, M^2\) holds. By the outer regularity of the capacity, \(\mu\) charges no set of zero capacity, and so \(\mu\) is smooth.

Without loss of generality, we may assume that \(\mu_n\) is a complete measure. By \cite[Theorem 3.3.14]{R87}, \(C_c\) is dense in \(L^1(E;m)\). By the duality of \(L^1\) and \(L^\infty\) spaces, we have
\begin{eqnarray*}
\|R_1|\mu_n-\mu|\|_\infty &=& \sup_{\|f\|_{L^1}\leq 1, f\in C_c} \int_E |f|\,R_1|\mu_n-\mu|\,dm\\
&=& \sup_{\|f\|_{L^1}\leq 1, f\in C_c} \int_E R_1|f|\,d|\mu_n-\mu|\\
&\leq & \sup_{\|f\|_{L^1}\leq 1, f\in C_c} \varliminf_{m\to \infty} \int_E R_1|f|\,d|\mu_n-\mu_m|\\
&\leq & \varliminf_{m\to \infty} \sup_{\|f\|_{L^1}\leq 1, f\in C_c}  \int_E R_1|f|\,d|\mu_n-\mu_m|\\
&=& \varliminf_{m\to \infty} \big\|R_1|\mu_n-\mu_m|\big\|_\infty.
\end{eqnarray*}
In the first inequality above, we used the assumption that \(R_1|f|\) is a lower semi-continuous and Lemma \ref{pot_lem}. Hence \(\|R_1|\mu_n-\mu|\|_\infty\) converges to \(0\) as \(n\to \infty\).
\end{proof}

In the proof of Theorem \ref{Dynkin_complete}, we have already seen the following proposition.
\begin{proposition}
For \(\mu_n, \mu \in \mathcal{D}\), if \(\mu_n\) converges to \(\mu\) in \(d_\mathcal{D}\), then so does vaguely. 
\end{proposition}

We consider the closedness of subsets of the Dynkin class.
\begin{theorem}\label{Kato_complete}
The Kato class \(\mathcal{K}\) is a closed subset of \((\mathcal{D}, d_{\mathcal{D}})\). In particular, if \(R_1f\) has a lower semi-continuous version for any non-negative function \(f\in C_c\), then \((\mathcal{K}, d_{\mathcal{D}})\) is a complete metric space.
\end{theorem}

\begin{proof}
We take \(\{\mu_n\}_n \subset \mathcal{K}\) and \(\mu\in\mathcal{D}\) such that \(d_\mathcal{D}(\mu_n, \mu)\) converges to \(0\) as \(n \to \infty.\) 

By the resolvent equation, for \(\alpha >1\), we have
\begin{eqnarray*}
\| R_{\alpha}|\mu_n-\mu| \|_{\infty}&=& \big\| R_1|\mu_n-\mu| -(\alpha -1)R_\alpha R_1 |\mu_n-\mu| \big\|_{\infty}\\
&\leq & \big\| R_1|\mu_n-\mu| \big\|_{\infty} + (\alpha-1) \| R_\alpha 1  \|_{\infty} \big\|R_1 |\mu_n-\mu| \big\|_{\infty}\\
&\leq & \big\| R_1|\mu_n-\mu| \big\|_{\infty}  + \alpha \| R_\alpha 1  \|_{\infty} \big\| R_1|\mu_n-\mu| \big\|_{\infty} \\
&\leq & 2 \big\| R_1|\mu_n-\mu| \big\|_{\infty}.
\end{eqnarray*}
For any \(\varepsilon >0\), we take large \(n\) satisfying \(d_{\mathcal{D}}(\mu_n, \mu)\leq \varepsilon/2\). Then we have
\begin{eqnarray*}
\| R_{\alpha}\mu \|_{\infty} &\leq & \| R_\alpha|\mu_n-\mu|\|_\infty + \| R_\alpha \mu_n\|_\infty \\
&\leq & 2\| R_1|\mu_n-\mu| \|_\infty  + \| R_\alpha \mu_n\|_\infty.
\end{eqnarray*}
By letting \(\alpha\) tend to infinity, we have 
\[\varlimsup_{\alpha \to \infty} \| R_{\alpha}\mu \|_{\infty} \leq \varepsilon\]
and so \(\mu \in \mathcal{K}.\)
\end{proof}

\begin{corollary}\label{Green-tight_complete}
The Green-tight class \(\mathcal{K}_\infty(R_1)\) is a closed subset of \((\mathcal{D}, d_{\mathcal{D}})\). In particular, if \(R_1f\) has a lower semi-continuous version for any non-negative function \(f\in C_c\), then \((\mathcal{K}_\infty(R_1), d_{\mathcal{D}})\) is a complete metric space.
\end{corollary}
\begin{proof}
We take \(\{\mu_n\}_n \subset \mathcal{K}_\infty(R_1)\) and \(\mu\in\mathcal{D}\) such that \(d_\mathcal{D}(\mu_n, \mu)\) converges to \(0\) as \(n \to \infty.\) By Theorem \ref{Kato_complete}, \(\mu\in \mathcal{K}\). For any \(\varepsilon>0\), we fix a large \(n\) satisfying \(\|R_1|\mu_n-\mu|\|_\infty \leq \varepsilon/2\) and we take a compact set \(K\) satisfying \(\|R_1(1_{K^c}\mu_n)\|_\infty \leq \varepsilon/2\). Then we have
\begin{equation*}
\|R_1(1_{K^c}\mu)\|_\infty \leq \|R_1|\mu_n-\mu|\|_\infty + \|R_1(1_{K^c}\mu_n)\|_\infty \leq \varepsilon.
\end{equation*}
\end{proof}

We recall that \(X\) satisfies the resolvent Feller property if \(R_1(C_\infty) \subset C_\infty\), and $X$ satisfies the resolvent strong Feller property if \(R_1(\mathcal{B}_b) \subset C_b\). Here, $\mathcal{B}_b$ denotes the set of all
bounded measurable functions, while $C_b$ and $C_\infty$ are the spaces of bounded continuous functions and continuous functions vanishing at infinity, respectively. The following follows immediately from Theorem \ref{Dynkin_complete}, \ref{Kato_complete} and Corollary \ref{Green-tight_complete}.

\begin{corollary}
If \(X\) enjoys the resolvent Feller property or the resolvent strong Feller property, then it holds that \(\mathcal{K}_\infty(R_1) \subset \mathcal{K} \subset \mathcal{D}\) and each subspace is complete with respect to \(d_{\mathcal{D}}\).
\end{corollary}

\begin{example}
In Example \ref{counterexofmetric1}, we considered \(1\)-order equilibrium measures \(\mu_n\) of \(B(0;1-\frac{1}{n})^c\) for the reflected Brownian motion on the unit disk \(B(0;1) \subset \mathbb{R}^2.\) We have already seen that \(\{U_1\mu_n\}_n\) is a Cauchy sequence in \(L^\infty\) but its limit is not a \(1\)-potential of any smooth measure. In this example, we see that \(\{\mu_n\}_n\) is not a Cauchy sequence with respect to \(d_\mathcal{D}\). We set the radial average \(\bar{r_1}\) by 
\[\bar{r_1}(r,s):= \frac{1}{2\pi}\int_0^{2\pi} r_1\left((r,0),(s\cos\theta, s\sin\theta)\right)\,d\theta\]
and then, by considering radial derivative, \(\bar{r_1}(r,s)\) is a Green resolvent of \(u(r)=\frac{1}{2}u''(r)+\frac{1}{2r}u'(r)\) with the Neumann condition at \(r=1\) whose homogeneous equation is the modified Bessel differential equation. Hence solutions to  \(u''(r)+\frac{1}{r}u'(r)-2u(r)=0\) are linear combinations of the \(0\)-order modified Bessel functions of the first kind \(I_0(\sqrt{2}r)\) and the second kind \(  K_0(\sqrt{2}r)\). Since \(u''(r)+\frac{1}{r}u'(r)-2u(r)=0\) is also a Sturm–Liouville equation, its Green function can be expressed by two solutions \cite[(5.65)]{T12} and so it holds that \(\bar{r}_1(r,s) = C (I_0(\sqrt{2}(r\wedge s)) + c_1 K_0(\sqrt{2}(r\wedge s)))(I_0(\sqrt{2}(r\vee s)) + c_2 K_0(\sqrt{2}(r\vee s))) \). Since \(K_0(0)=\infty,\) we have \(c_1=0\). By the Neumann condition, we have \(c_2=I_1(\sqrt{2})/K_1(\sqrt{2})\)  where \(I_1\) and \(K_1\) are the \(1\)-order modified Bessel functions of the first and second kind, respectively. 
We set \(a_n:=-I_0'(\sqrt{2}(1-\frac{1}{n}))/(\sqrt{2}\pi I_0(\sqrt{2}(1-\frac{1}{n})))\) and prove that \(\nu_n:= \pi^{-1}1_{B(0;1-1/n)^c}dx + a_n \sigma_{\partial B(0;1-\frac{1}{n})}\) is a \(1\)-equilibrium measure on \(B(0;1-\frac{1}{n})^c\), where \(\sigma_{\partial B(0;1-\frac{1}{n})}\) is a Lebesgue measure on \(\partial B(0;1-\frac{1}{n})\). Let \(u(r):=I_0(\sqrt{2}r) + c_2 K_0(\sqrt{2}r)\) and by using the rotational symmetry, when \(|x|=r\leq 1-\frac{1}{n}\), we have
\begin{eqnarray*}
\int r_1(x,y)\,d\nu_n(y) &=& 2 \int_{1-\frac{1}{n}}^1 \bar{r_1}(r,s)s\,ds + 2\pi(1-\frac{1}{n}) a_n \bar{r_1}(r,1-\frac{1}{n})\\
&=& 2 C I_0(\sqrt{2}r) \left( \int_{1-\frac{1}{n}}^1 u(s)s\,ds+\pi a_n(1-\frac{1}{n}) u(1-\frac{1}{n}) \right)\\
&=& 2 C I_0(\sqrt{2}r) (1-\frac{1}{n}) \left(-\frac{1}{2}u'(1-\frac{1}{n}) + \pi a_n u(1-\frac{1}{n}) \right)\\
&=& \frac{I_0(\sqrt{2}r)}{I_0(\sqrt{2}(1-\frac{1}{n}))}.
\end{eqnarray*}
In the second inequality above, we used the relations \((su')'=2su\) and \(u'(1)=0\). For the third
inequality, we applied the fact that the Wronskian of $u(r)$ and $I_0(\sqrt{2} r)$ taken with respect to $r$ is \((C \sqrt{2} r)^{-1}\) ( \cite[(5.65)]{T12} ). Similarly, we obtain \(\int r_1(x,y)\,d\nu_n(y) =1\) when \(1-\frac{1}{n}<|x|=r<1\), hence, the 1-resolvent of \(\nu_n\) coincides with the 1-order equilibrium potential \(u_n\) in Example \ref{counterexofmetric1}. Then \(\nu_n\) is a \(1\)-equilibrium measure \(\mu_n\) on \(B(0;1-\frac{1}{n})^c\). 

For \(n\geq m\), we have 
\begin{eqnarray*}
|\mu_n-\mu_m|&=&\left|  \pi^{-1}1_{\{1-\frac{1}{m} \leq |x| < 1-\frac{1}{n}\}}\,dx + a_m\sigma_{\partial B(0;1-\frac{1}{m})} -a_n\sigma_{\partial B(0;1-\frac{1}{n})} \right|\\
&=&  \pi^{-1}1_{\{1-\frac{1}{m} \leq |x| < 1-\frac{1}{n}\}}\,dx + a_m\sigma_{\partial B(0;1-\frac{1}{m})} +a_n\sigma_{\partial B(0;1-\frac{1}{n})}\\
&\geq & a_n\sigma_{\partial B(0;1-\frac{1}{n})}.
\end{eqnarray*}
Consequently, 
\begin{eqnarray*}
\|U_1|\mu_n-\mu_m|\|_\infty &\geq & a_n \| U_1\sigma_{\partial B(0;1-\frac{1}{n})}\|_\infty\\
&=& \sup_{0\leq r<1}  \sqrt{2} C (1-\frac{1}{n}) \frac{I_1(\sqrt{2}(1-\frac{1}{n}))}{I_0(\sqrt{2}(1-\frac{1}{n}))} I_0(\sqrt{2}r) u(1-\frac{1}{n})\\
&\geq & \sqrt{2} C (1-\frac{1}{n}) I_1(\sqrt{2}(1-\frac{1}{n})) u(1-\frac{1}{n})\\
&\to &  \sqrt{2} C I_1(\sqrt{2}) u(1)
\end{eqnarray*}
as $n\to \infty$. Since the last term is positive, it follows that $\{\mu_n\}_n$ is not a Cauchy
sequence with respect to $d_{\cal D}$.
\end{example}

Next we consider the convergence of PCAFs and Feynman-Kac type semigroups. We give the following decomposition of potentials. 
\begin{proposition}\label{L1Fukushima_decomp}
Fix \(\alpha >0\). For \(\mu \in \mathcal{S}\) with \(R_\alpha \mu (x) < \infty\) for q.e. \(x\in E\), there exists an additive functional \(M\) such that \(M\) is an \(L^1(\mathbb{P}_x)\)-martingale whose expectation is \(0\) q.e. \(x\in E\), and
\[R_{\alpha}\mu(X_t)-R_{\alpha}\mu(X_0)=M_t +\alpha \int_0^t R_{\alpha}\mu(X_s) ds -A_t^\mu \ \ \mathbb{P}_x{\text-a.s. \ for\ q.e.} x\in E {\text \ and\ } t\geq 0.\]
\end{proposition}

\begin{proof}
For convenience, we define \(\widetilde{A^\mu_t}:= \int_0^t e^{-\alpha s}dA_s^\mu\). Since \(\mathbb{E}_x[\widetilde{A^\mu_\infty}]=R_\alpha \mu (x) <\infty \), \(\widetilde{A^\mu_t}\) is a \([0,\infty)\)-valued continuous process, and it satisfies that
\[\widetilde{A_{t+s}^\mu}=\widetilde{A_{t}^\mu}+ e^{-\alpha t} \widetilde{A^\mu_{s}}\circ \theta_t.\]
Hence we have
\[R_\alpha \mu(X_t)= \mathbb{E}_{X_t}[\widetilde{A_{\infty}^\mu}]=\mathbb{E}_{x}[\widetilde{A_{\infty}^\mu}\circ \theta_t |\mathcal{F}_t]=e^{\alpha t}\mathbb{E}_{x}[\widetilde{A_{\infty}^\mu} |\mathcal{F}_t] -e^{\alpha t}\widetilde{A_t^\mu}.\]

We set
\[M_t:= R_{\alpha}\mu(X_t)-R_{\alpha}\mu(X_0)-\alpha \int_0^t R_{\alpha}\mu(X_s) ds + A_t^\mu.\]
For q.e. \(x\in E\), we have
\begin{eqnarray*}
M_t &=& e^{\alpha t}\mathbb{E}_{x}[\widetilde{A_{\infty}^\mu} |\mathcal{F}_t] -e^{\alpha t}\widetilde{A_t^\mu}-\mathbb{E}_{x}[\widetilde{A_{\infty}^\mu}]-\int_0^t \alpha e^{\alpha s}\left(\mathbb{E}_{x}[\widetilde{A^\mu_{\infty}} |\mathcal{F}_s] -\widetilde{A_s^\mu} \right) ds + A_t^\mu \\
&=&  e^{\alpha t}\mathbb{E}_{x}[\widetilde{A_{\infty}^\mu} |\mathcal{F}_t] -e^{\alpha t}\widetilde{A_t^\mu}-\mathbb{E}_{x}[\widetilde{A_{\infty}^\mu}]-\int_0^t \alpha e^{\alpha s}\mathbb{E}_{x}[\widetilde{A_{\infty}^\mu} |\mathcal{F}_s]ds+
\int_0^t \int_0^s \alpha e^{\alpha s}e^{-\alpha u}dA_u^\mu ds + A_t^\mu \\
&=&  e^{\alpha t}\mathbb{E}_{x}[\widetilde{A_{\infty}^\mu} |\mathcal{F}_t] -e^{\alpha t}\widetilde{A_t^\mu}-\mathbb{E}_{x}[\widetilde{A_{\infty}^\mu}]-\int_0^t \alpha e^{\alpha s}\mathbb{E}_{x}[\widetilde{A_{\infty}^\mu} |\mathcal{F}_s]ds+
\int_0^t \int_u^t \alpha e^{\alpha s}e^{-\alpha u}dsdA_u^\mu  + A_t^\mu \\
&=&  e^{\alpha t}\mathbb{E}_{x}[\widetilde{A_{\infty}^\mu} |\mathcal{F}_t] -e^{\alpha t}\widetilde{A_t^\mu}-\mathbb{E}_{x}[\widetilde{A_{\infty}^\mu}]-\int_0^t \alpha e^{\alpha s}\mathbb{E}_{x}[\widetilde{A_{\infty}^\mu} |\mathcal{F}_s]ds+
\int_0^t (e^{\alpha(t-u)}-1)dA_u^\mu  + A_t^\mu \\
&=&  e^{\alpha t}\mathbb{E}_{x}[\widetilde{A_{\infty}^\mu} |\mathcal{F}_t] -\mathbb{E}_{x}[\widetilde{A_{\infty}^\mu}]-\int_0^t \alpha e^{\alpha s}\mathbb{E}_{x}[\widetilde{A_{\infty}^\mu} |\mathcal{F}_s]ds
\end{eqnarray*}
and thus we get
\[\mathbb{E}_x[M_t]= e^{\alpha t}\mathbb{E}_{x}[\widetilde{A_{\infty}^\mu}] -\mathbb{E}_{x}[\widetilde{A_{\infty}^\mu}]-\int_0^t \alpha e^{\alpha s}\mathbb{E}_{x}[\widetilde{A_{\infty}^\mu}]ds =0.\]
Moreover we have
\begin{eqnarray*}
\mathbb{E}_x[M_{t+s}|\mathcal{F}_t] &=& e^{\alpha (t+s)} \mathbb{E}_x[\mathbb{E}_{x}[\widetilde{A_{\infty}^\mu} |\mathcal{F}_{t+s}]|\mathcal{F}_t]  -\mathbb{E}_{x}[\mathbb{E}_{x}[\widetilde{A_{\infty}^\mu}]|\mathcal{F}_t] -\int_0^{t+s} \alpha e^{\alpha u}\mathbb{E}_{x}[\mathbb{E}_{x}[\widetilde{A_{\infty}^\mu} |\mathcal{F}_u]|\mathcal{F}_t] du\\
&=& e^{\alpha (t+s)} \mathbb{E}_x[\widetilde{A_{\infty}^\mu} |\mathcal{F}_t]  -\mathbb{E}_{x}[\widetilde{A_{\infty}^\mu}] -\int_0^{t} \alpha e^{\alpha u}\mathbb{E}_{x}[\widetilde{A_{\infty}^\mu} |\mathcal{F}_u] du-\int_t^{t+s} \alpha e^{\alpha u}\mathbb{E}_{x}[\widetilde{A_{\infty}^\mu}|\mathcal{F}_t] du\\
&=& e^{\alpha (t+s)} \mathbb{E}_x[\widetilde{A_{\infty}^\mu} |\mathcal{F}_t]  -\mathbb{E}_{x}[\widetilde{A_{\infty}^\mu}] -\int_0^{t} \alpha e^{\alpha u}\mathbb{E}_{x}[\widetilde{A_{\infty}^\mu} |\mathcal{F}_u] du- e^{\alpha t}(e^{\alpha s}-1) \mathbb{E}_{x}[\widetilde{A_{\infty}^\mu}|\mathcal{F}_t] \\
&=&  M_t.
\end{eqnarray*}
Hence \(M\) is an \(L^1(\mathbb{P}_x)\)-martingale.
\end{proof}

\begin{remark}
If \(\mathbb{E}_x[(\widetilde{A_{\infty}^\mu})^2] <\infty\), the above \(M_t\) can be represented by the stochastic integral
\[M_t = \int_0^t e^{\alpha s} d\,\mathbb{E}_{X_0}[\widetilde{A_{\infty}^\mu}|\mathcal{F}_s].\]
\end{remark}

\begin{theorem}
For \(\mu_n, \mu \in \mathcal{D}\) and their corresponding PCAFs \(A^{\mu_n}, A^{\mu}\), if \(\mu_n\) converges to \(\mu\) in \(d_\mathcal{D}\), then \(A^{\mu_n}\) converges to \(A^{\mu}\) in \(L^1(\mathbb{P}_x)\) uniformly for \(x\) with the local uniform topology for \(t\), that is, for any \(T>0\),
\[\lim_{n\to \infty} \left\|\mathbb{E}_\cdot[\sup_{0\leq t\leq T} |A^{\mu_n}_t-A^{\mu}_t|]\right\|_{\infty,q} = 0.\]
\end{theorem}
\begin{proof}
Let \(A^{|\mu_n-\mu|}\) be a PCAF corresponding to \(|\mu_n-\mu| \in \mathcal{D}\). By Proposition \ref{L1Fukushima_decomp}, we have, for any \(t\),
\[P_tR_1|\mu_n-\mu| - R_1|\mu_n-\mu| = \int_0^t P_tR_1|\mu_n-\mu|\,dt - \mathbb{E}_x[A^{|\mu_n-\mu|}_t]\ {\text \ q.e.\ }\]
Hence \(\| \mathbb{E}_\cdot[A^{|\mu_n-\mu|}_t]\|_{\infty, q}\) converges to \(0\) as \(n\) tends to infinity. By the uniqueness of a PCAF, we can take an appropriate exceptional set \(N_n\) such that \(A^{|\mu_n-\mu|}_t = A^{(\mu_n-\mu)_+}_t + A^{(\mu_n-\mu)_-}_t\) and \(A^{\mu_n}_t-A^{\mu}_t = A^{(\mu_n-\mu)_+}_t - A^{(\mu_n-\mu)_-}_t\) \(\mathbb{P}_x\)-almost surely for \(x \in N^c_n\). Therefore we have
\[\sup_{0\leq t\leq T} |A^{\mu_n}_t-A^{\mu}_t| \leq  \sup_{0\leq t\leq T} A^{|\mu_n-\mu|}_t =  A^{|\mu_n-\mu|}_T\]
 \(\mathbb{P}_x\)-almost surely for \(x \in N^c_n\), and so \(A^{\mu_n}\) converges to \(A^{\mu}\) in \(L^1(\mathbb{P}_x)\) uniformly for \(x\) with the local uniform topology for \(t\).
\end{proof}

\begin{remark}
In general, a PCAF may blow up within finite time.  However, for a smooth measure $\mu \in{\cal D}$, the PCAF $A^\mu$ satisfies \(\mathbb{P}_x(A_T^\mu <\infty)=0\), so we can consider the local uniform topology.
\end{remark}

As for the locally uniform convergence  of additive functionals in terms of the Miyadera norm, similar arguments  developed in \cite{NTTU25} tell us the following proposition:

\begin{proposition} Let $\{\mu_n, \mu\} \subset {\cal D}$ and let $\{A^n, A\}$ be the corresponding PCAFs.
Assume that $d_{\cal D}(\mu_n, \mu) \to 0$ as $n \to \infty$. Then the following assertions hold:
\begin{itemize}
\item[\rm (1)] For any compact set $K \subset E$, there exists a subsequence $\{n_k\}$ such that the PCAFs $A^{n_k}$ converge to $A$ locally uniformly in $t$ on $[0,\infty)$ restricted to $K$, ${\mathbb P}_x$-a.s. for quasi every $x \in E$. That is,
$$
{\mathbb P}_x \left( \lim_{k \to \infty} \sup_{0 \le s \le t} \left| \int_0^s 1_K(X_u) \, dA^{n_k}_u - \int_0^s 1_K(X_u) \, dA_u \right| = 0 \text{ for all } t > 0 \right) = 1, \quad \text{q.e. } x \in E.
$$

\item[\rm (2)] Suppose there exists a compact nest $\{F_\ell\}$ satisfying the following tail estimate:
\begin{equation} \label{eqn:tail}
\lim_{\ell \to \infty} \sup_{\nu \in {\cal S}_{c,00}^1} \left\{ \left( \sup_n \int_{F_\ell^c} U_1 \nu \, d\mu_n \right) \vee \int_{F_\ell^c} U_1 \nu \, d\mu \right\} = 0.
\end{equation}
Then, there exists a further subsequence $\{n'_k\}$ such that $A^{n'_k}$ converges to $A$ locally uniformly in $t$ on $[0,\infty)$, ${\mathbb P}_x$-a.s. for quasi every $x \in E$:
$$
{\mathbb P}_x \left( \lim_{k \to \infty} \sup_{0 \le s \le t} |A^{n'_k}_s - A_s| = 0 \text{ for all } t > 0 \right) = 1, \quad \text{q.e. } x \in E,
$$
where ${\cal S}_{c,00}^1 := \{ \nu \in {\cal S}_0 : \mathrm{supp}[\nu] \text{ is compact}, \nu(E) \le 1 \text{ and } \|U_1\nu\|_\infty \le 1 \}$.
\end{itemize}
\end{proposition}

\noindent
{\it Proof}:   We give the proof of (2) only. Assertion (1) follows directly from (2) by replacing the measures $\mu_n$ and $\mu$ with the restricted measures $1_K \mu_n$ and $1_K \mu$, respectively, and choosing a compact nest $\{F_\ell\}$ such that $F_\ell \uparrow E$, and for each compact set $K$ there exists an $\ell$ satisfying $K\subset F_\ell$. Indeed, for such an exhausting nest, we have $K \subset F_\ell$ for all sufficiently large $\ell$. This implies that $F_\ell^c \cap K = \varnothing$, which makes the tail estimate \eqref{eqn:tail} trivially zero for such $\ell$.

First, note that $\{1_{F_\ell} \mu_n, 1_{F_\ell}\mu\}$ belong to ${\cal S}_0$ for each $\ell$. 
We shall estimate the following probability by decomposing it into three parts, following the argument in the proof of Theorem 5.1 in \cite{NTTU25}:
\begin{align*}
{\mathbb P}_\nu\Big( \sup_{0\le s \le t} |A^n_s-A_s| > \vareps \Big) 
& \le  {\mathbb P}_\nu\Big( \sup_{0\le s \le t} |A^n_s-(1_{F_\ell}A^n)_s| > \frac{\vareps}3 \Big) 
  +  {\mathbb P}_\nu\Big( \sup_{0\le s \le t} |(1_{F_\ell}A^n)_s-(1_{F_\ell}A)_s| > \frac{\vareps}3 \Big)\\
  & \qquad  +  {\mathbb P}_\nu\Big( \sup_{0\le s \le t} |(1_{F_\ell}A)_s- A_s| > \frac{\vareps}3 \Big) \\
  & =: {\sf (I)} + {\sf (II)} + {\sf (III)}
 \end{align*}
for any $\vareps>0$, $t>0$, $\ell \ge 1$  and $\nu \in {\cal S}_{c,00}^1$.

{\sf step 1.} For the terms {\sf (I)} and {\sf (III)},  using the Revuz  correspondence, the following estimates hold:
$$
{\sf (I)} \le \frac 3{\vareps} {\mathbb E}_\nu\Big[ \int_0^t 1_{F_\ell^c} dA_s^n\Big] \le 
\frac{3e^t}{\vareps} \int_{F_\ell^c} U_1\nu d\mu_n
\quad {\rm and} \quad  {\sf (III)} \le  \frac{3e^t}{\vareps} \int_{F_\ell^c} U_1\nu d\mu.
$$
Then, by the tail estimate \eqref{eqn:tail}, the terms {\sf (I)} and {\sf (III)} can be made arbitrarily small uniformly in $n$ by taking $\ell$ 
sufficiently large.

{\sf step 2.} For the term {\sf (II)}, it follows from Remark \ref{rem:3.4} (1) and the fact that $\nu \in {\cal S}_{c,00}^1$ that
$$
{\sf (II)} \le \frac{3e^t}{\vareps} \int_{F_\ell} U_1\nu \, d|\mu_n-\mu|
\le \frac{3e^t}{\vareps} |\mu_n-\mu|(F_\ell) 
\le \frac{3e^t}{\vareps}  {\rm Cap}(F_\ell) d_{\cal D}(\mu_n, \mu). 
$$
Since $\ell$ is fixed, the right-hand side converges to $0$ as $n \to \infty$ by the assumption $d_{\cal D}(\mu_n, \mu) \to 0$.

{\sf step 3.}  Combining {\sf step 1.} and {\sf step 2.},  we have shown 
$$
\limsup_{n\to \infty} \sup_{\nu \in {\cal S}_{c,00}^1} 
{\mathbb P}_\nu\Big( \sup_{0\le s\le t} |A_s^n-A_s|>\vareps \Big)=0.
$$
Then, by extracting a suitable subsequence $\{n'_k\}$, it follows that 
\begin{align*}
\sup_{\nu \in {\cal S}_{c,00}^1} {\mathbb P}_\nu \left( \sup_{0 \le s \le t} |A^{n'_k}_s - A_s| > 2^{-k} \right) 
\le 2^{-k}.
\end{align*}
Finally, taking the sum over $k$ and applying the first Borel-Cantelli lemma together with \cite[Theorem2.2.3]{FOT11} and \cite[Theorem 2.1]{NTTU25}, we can conclude the locally uniform convergence ${\mathbb P}_x$-a.s. for q.e. $x \in E$. \qed

We have already seen that considering a topology that is weaker and simpler than the Miyadera norm \(d_\mathcal{D}\) results in the loss of completeness. However, the Miyadera norm is a topology that requires a strong convergence. Therefore, at the end of this section, we propose a different topology as follows.

In \cite{NTTU25}, the complete separable metric \(d_{\mathcal{S}_0}\) on \(\mathcal{S}_0\) is defined by \[d_{\mathcal{S}_0}(\mu, \nu):=\sqrt{\mathcal{E}_1(U_1\mu-U_1\nu, U_1\mu-U_1\nu )}\] for \(\mu, \nu \in \mathcal{S}_0\), and a homeomorphism of the Revuz correspondence restricted to \(\mathcal{S}_0\) is established in \cite{O26} by introducing the topology on the space of PCAFs from the perspective of killing of a process. In \cite{OTU25}, metrics on a class \(\mathcal{S}_0(\{F_k\})\) of smooth measures attached to a compact nest \(\{F_k\}\) are introduced. More precisely, for a compact nest \(\{F_k\}\), \(\mathcal{S}_0(\{F_k\})\) is the set of all smooth measures \(\mu\) such that \(1_{F_k}\mu \in \mathcal{S}_0\) for each \(k\), and 
\[d_{\mathcal{S}_0(\{F_k\})}(\mu, \nu) := \sum_{k=1}^\infty \frac{1}{2^k} \left(1 \wedge d_{\mathcal{S}_0}(1_{F_k}\mu, 1_{F_k}\nu) \right)\]
for \(\mu, \nu \in \mathcal{S}_0(\{F_k\}).\) By \cite[Lemma 3.1]{OTU25}, \((\mathcal{S}_0(\{F_k\}), d_{\mathcal{S}_0(\{F_k\})})\) is a complete separable metric space.

Similarly, for a compact nest \(\{F_k\}\), we set \(\mathcal{S}_{00}(\{F_k\})\) the class of all smooth measure \(\mu\) such that \(1_{F_k}\mu \in \mathcal{S}_{00}\) for each \(k\). Denote by \(\nest\) the set of all compact nest, and then, by Theorem \ref{Thm_col}, we have
\[\mathcal{S}=\bigcup_{\{F_k\} \in \nest}\mathcal{S}_0(\{F_k\}) = \bigcup_{\{F_k\} \in \nest}\mathcal{S}_{00}(\{F_k\}).\]

Since \(\mathcal{S}_{00}(\{F_k\})\) is not closed in \((\mathcal{S}_{0}(\{F_k\}), d_{\mathcal{S}_{0}(\{F_k\})})\), we define another metric \(d_{\mathcal{S}_{00}(\{F_k\})}^{\, \infty}\) on \(\mathcal{S}_{00}(\{F_k\})\) by
\[d_{\mathcal{S}_{00}(\{F_k\})}^{\, \infty} (\mu, \nu) := \sum_{k=1}^\infty \frac{1}{2^k} \left(1 \wedge \| U_1(1_{F_k}\mu)-U_1(1_{F_k}\nu) \|_\infty \right)\]
for \(\mu, \nu \in \mathcal{S}_{00}(\{F_k\}).\)

\begin{theorem}\label{completeSnest}
\((\mathcal{S}_{00}(\{F_k\}), d_{\mathcal{S}_{00}(\{F_k\})}^\infty)\) is a complete  metric space.
\end{theorem}

\begin{proof}[Proof] 
We take a \(d_{\mathcal{S}_{00}(\{F_k\})}^\infty\)-Cauchy sequence $\{\mu_n\}_n \subset \mathcal{S}_{00}(\{F_k\})$. Then, for each \(k \in \mathbb{N}\), \(\{U_1(1_{F_k}\mu_n)\}_n \) is a \(\|\cdot\|_{\infty}\)-Cauchy sequence. Hence there exists a function $v_k$ defined quasi-everywhere on \(E\) such that
\[
\lim_{n\to \infty} \left\| U_1(1_{F_k}\mu_n)-v_k \right\|_{\infty}=0.
\] 
Moreover, {
  since $1_{F_k}\mu_n \in \mathcal{S}_{00}$, it follows from Remark \ref{rem:3.4} (1) that 
$$
\sup_n \mu_n(F_k) \le  {\rm Cap}(F_k) \| U_1(1_{F_k}\mu_n)\|_\infty.
$$
Since $\{U_1(1_{F_k}\mu_n)\}_n$ is a $\|\cdot\|_\infty$-Cauchy sequence, its supremum norm is uniformly bounded in $n$. Thus we obtain $\sup_n \mu_n(F_k)<\infty$. Furthermore, applying Remark  \ref{rem:3.4}(1) again to the total variation measure $1_{F_k}|\mu_n - \mu_m| \in \mathcal{S}_{00}$, we have
\begin{align*}
d_{\mathcal{S}_0}(1_{F_k}\mu_n, 1_{F_k}\mu_m)^2 &\le \mathcal{E}_1(U_1(1_{F_k}|\mu_n - \mu_m|), U_1(1_{F_k}|\mu_n - \mu_m|)) \\
&\le |\mu_n - \mu_m|(F_k) \| U_1(1_{F_k}|\mu_n - \mu_m|) \|_\infty \\
&\le 2 \left( \sup_l \mu_l(F_k) \right) \| U_1(1_{F_k}|\mu_n - \mu_m|) \|_\infty.
\end{align*}
Since $\sup_l \mu_l(F_k) < \infty$ and $\{\mu_n\}_n$ is a $d_{\mathcal{S}_{00}(\{F_k\})}^\infty$-Cauchy sequence, the right-hand side converges to $0$ as $n, m \to \infty$. Hence $\{1_{F_k}\mu_n\}_n$ is a $d_{\mathcal{S}_0}$-Cauchy sequence in $\mathcal{S}_0$ for each $k$, and so $\{\mu_n\}_n$ is a $d_{\mathcal{S}_0(\{F_k\})}$-Cauchy sequence. Therefore, there exists $\mu \in \mathcal{S}_0(\{F_k\})$ such that $\mu_n$ converges to $\mu$ in $d_{\mathcal{S}_0(\{F_k\})}$.

Next, we prove that $U_1(1_{F_k}\mu)=v_k$ quasi-everywhere and $\mu \in \mathcal{S}_{00}(\{F_k\})$. Since $1_{F_k}\mu_n$ converges to $1_{F_k}\mu$ in $d_{\mathcal{S}_0}$ as $n\to \infty$, their 1-potentials $U_1(1_{F_k}\mu_n)$ converge to $U_1(1_{F_k}\mu)$ with respect to the $\mathcal{E}_1$-norm. By taking a suitable subsequence $n'$, a quasi-continuous version of $U_1(1_{F_k}\mu_{n'})$ converges to a quasi-continuous version of $U_1(1_{F_k}\mu)$ quasi-everywhere (cf.\ \cite[Theorem 2.1.4]{FOT11}). Since $U_1(1_{F_k}\mu_{n'})$ also converges to $v_k$ uniformly, we have $U_1(1_{F_k}\mu) =v_k$ quasi-everywhere. Moreover, $\|U_1(1_{F_k}\mu)\|_{\infty}=\|v_k\|_{\infty}<\infty$ holds, so $\mu \in \mathcal{S}_{00}(\{F_k\})$.
}
\end{proof}

\begin{proposition}
For \(\mu_n, \mu \in \mathcal{S}_{00}(\{F_k\})\), if \(\mu_n\) converges to \(\mu\) in \(d_{\mathcal{S}_{00}(\{F_k\})}^\infty\), then \(\mu_n\) converges to \(\mu\) weakly on the nest \(\{F_k\}\), that is, for any \(f \in C(\{F_k\})\) and \(F_k\), \(\int_{F_k} f\, d\mu_n\) converges to \(\int_{F_k} f\, d\mu\).
\end{proposition}

\begin{proof}
  Assume that \(\mu_n\) converges to \(\mu\) in \(d_{\mathcal{S}_{00}(\{F_k\})}^\infty\). Then, for any $k \in \mathbb{N}$, \(U_1(1_{F_k} \mu_n)\) converges to \(U_1(1_{F_k} \mu)\) in \(\|\cdot\|_\infty\). {
    By  Remark \ref{rem:3.4} (1)}, we find that \(M_k:=\sup_n \mu_n(F_k) \vee \mu(F_k)<\infty.\)

Let $f\in C(\{F_k\})$ and $k \in \mathbb{N}$. Since $F_k$ is compact and $f$ is continuous on $F_k$, we can find a function $g\in C_c(E)$ with $f=g$ on $F_k$. By the regularity, for each $\vareps>0$, there exists $h \in \mathcal{F}\cap C_c(E)$ such that 
$\|g-h\|_\infty <\vareps/2(M_k +1)$.

Then we have
\begin{align*}
\lefteqn{
\Big| \int_{F_k} f(x) \mu_n(dx) - \int_{F_k} f(x) \mu(dx) \Big|   \ 
 =\Big| \int_{F_k} g(x) \mu_n(dx) - \int_{F_k} g(x) \mu(dx) \Big|}   \\
&  \le  \int_{F_k} \big|g(x)-h(x)\big| \mu_n(dx)
+ \Big| \int_{F_k} h(x) \mu_n(dx) - \int_{F_k} h(x) \mu(dx) \Big|  
+ \int_{F_k} \big|h(x)-g(x) \big| \mu(dx) \\
& \le 2 M_k \|g-h\|_\infty 
+ \Big| \mathcal{E}_1(h, U_1(1_{F_k}\mu_n)) -  \mathcal{E}_1(h, U_1(1_{F_k}\mu))\Big|  \\
& \le \vareps + \sqrt{\mathcal{E}_1(h,h)} \sqrt{\mathcal{E}_1(U_1(1_{F_k}\mu_n)-U_1(1_{F_k}\mu), 
U_1(1_{F_k}\mu_n)-U_1(1_{F_k}\mu))} \\
& = \vareps + \sqrt{\mathcal{E}_1(h,h)} \sqrt{ 
\int_{F_k} \hspace*{-5pt} 
\big(U_1(1_{F_k}\mu_n)-U_1(1_{F_k}\mu)\big)\mu_n(dx)-
\int_{F_k}\big(U_1(1_{F_k}\mu_n)-U_1(1_{F_k}\mu)\big)\mu(dx)} \\
& \le \vareps + \sqrt{\mathcal{E}_1(h,h)}  \sqrt{ 2M_k \big\| 
U_1(1_{F_k}\mu_n)-U_1(1_{F_k}\mu)\big\|_{\infty}}.
\end{align*}
By tending $n\to \infty$ and then $\vareps \to 0$, the last term goes to $0$.  
\end{proof}

For a compact nest $\{F_k\}$, denote by  $\mathcal{K}(\{F_k\})$ the set of all smooth measures satisfying $1_{F_k} \mu \in \mathcal{K}$ for each $k$. Then, by Lemma \ref{K00isS00}, it holds that
\[\mathcal{K}_{loc} \subset \mathcal{K}(\{F_k\}) \subset \mathcal{S}_{00}(\{F_k\}).\] 
By Theorem \ref{Thm_col}, we obtain
 \[\mathcal{S}=\bigcup_{\{F_k\} \in \nest}\mathcal{S}_0(\{F_k\}) = \bigcup_{\{F_k\} \in \nest}\mathcal{S}_{00}(\{F_k\}) =  \bigcup_{\{F_k\} \in \nest}\mathcal{K}(\{F_k\}).\]
Moreover, we get the following proposition.

\begin{proposition}
For each compact nest $\{F_k\}$, $\mathcal{K}(\{F_k\})$ is a closed subset of \((\mathcal{S}_{00}(\{F_k\}), d_{\mathcal{S}_{00}(\{F_k\})}^\infty)\).
\end{proposition}

\begin{proof}
We take a \(d_{\mathcal{S}_{00}(\{F_k\})}^\infty\)-Cauchy sequence $\{\mu_n\}\subset {\cal K}(\{F_k\})$, then, by Theorem \ref{completeSnest}, there exist \(\mu \in \mathcal{S}_{00}(\{F_k\})\) such that \(\mu_n\) converges to \(\mu\) in \(d_{\mathcal{S}_{00}(\{F_k\})}^\infty\). For each $k \in \mathbb{N}$, \(U_1(1_{F_k}\mu_n)\) converges to \(U_1(1_{F_k}\mu)\) in \(\|\cdot\|_\infty\). By the resolvent equation, we have
\begin{eqnarray*}
\big\| U_\alpha (1_{F_k} \mu) \big\|_{\infty} &\leq & \big\| U_\alpha (1_{F_k} \mu)-U_\alpha (1_{F_k} \mu_n) \big\|_{\infty} + \big\| U_\alpha (1_{F_k} \mu_n) \big\|_{\infty} \\
&\leq & \big\| U_1 (1_{F_k} \mu)-U_1 (1_{F_k} \mu_n)\big\|_{\infty}  \\
&&\hspace{10mm}+ \frac{\alpha-1}{\alpha}  
\big\| \alpha R_\alpha (U_1 (1_{F_k} \mu)-U_1 (1_{F_k} \mu_n))\big\|_{\infty} + \big\| U_\alpha (1_{F_k} \mu_n) \big\|_{\infty} \\
& \leq & 2 \big\| U_1 (1_{F_k} \mu)-U_1 (1_{F_k} \mu_n) \big\|_{\infty} 
+\big\| U_\alpha (1_{F_k} \mu_n) \big\|_{\infty}
\end{eqnarray*}
Letting \(\alpha\) tend to infinity and then \(n\) tend to infinity, \(\lim_{\alpha \to \infty}\| U_\alpha (1_{F_k}\mu)\|_{\infty}=0\).
\end{proof}

\appendix
\section{Basic definitions of Dirichlet form theory} \label{sec_Appndeix}
This appendix is devoted to providing the necessary background on Markov processes and Dirichlet forms used in this paper,
following the notation in \cite{FOT11, CF12}.

Let $E$ be a locally compact separable metric space and $m$ a positive Radon measure on $E$ with full topological support. We set $E_{\partial} = E \cup \{\partial\}$ by adding the cemetery point \(\partial \not \in E\) to \(E\).

Throughout this paper, denote by \((\mathcal{E}, \mathcal{F})\) a regular Dirichlet form on \(L^2(E;m)\) with inner product $(\cdot,\cdot)$. More precisely, \(\mathcal{E}\) is a non-negative definite symmetric bilinear form on a dense linear subspace \(\mathcal{F}\subset L^2(E;m)\) such that \((\mathcal{F} ,\mathcal{E}_1) \) is a Hilbert space, \(\bar{f}\, (:= 0 \vee f \wedge 1) \in \mathcal{F}\) and \(\mathcal{E}(\bar{f},\bar{f}) \leq \mathcal{E}(f,f)\) hold for any \(f\in \mathcal{F}\), and \(\mathcal{F}\cap C_c\) is dense in \(\mathcal{F}\) with respect to the \(\mathcal{E}_1\)-norm and \(C_c\) with respect to the uniform norm \(\|\cdot \|_\infty\), where \(C_c\) is the space of all continuous functions with compact support. Here, we define
\begin{align}
 \c{E}_{\alpha}(f,g) := \c{E}(f,g) + \alpha (f,g)
\end{align}
for $\alpha > 0$ and $u,v \in \mathcal{F}$, 
and the \(\mathcal{E}_1\)-norm of $f$ by \(\sqrt{\mathcal{E}_1(f,f)}\).
For the convenience, we set \(\mathcal{E}_{\alpha}(f):=\mathcal{E}_{\alpha}(f,f)\) for \(\alpha \geq 0\) and \(f\in \mathcal{F}\).

By the general theory of Dirichlet forms (cf. \cite{FOT11}), there exists an $m$-symmetric Hunt process $\mathbb{M} = (\Omega, \mathcal{F}_t,X_t,\mathbb{P}_x,\zeta)$ on $E$ associated with \((\mathcal{E}, \mathcal{F})\). More precisely, $\mathbb{M}$ is an \(E\)-valued strong Markov process on a probability space \(\Omega\) satisfying a quasi-left continuity, \(\mathbb{P}_x\) is a probability measure with starting point \(x\), $\{\mathcal{F}_t\}_{t \ge 0}$ is the minimal augmented filtration and $\zeta=\inf\{t\geq 0 : X_t=\partial\}$ is the lifetime of \(\mathbb{M}\).
Moreover, let $\{T_{t}\}_{t > 0}$ be the strongly continuous Markovian
  semigroup on $L^{2}(E;m)$ induced by the transition semigroup \(\{P_t\}_{t>0}\) defined by \(P_{t} f(x) := \mathbb{E}_x[f(X_t)]\). Then it holds that
\begin{align}
 \begin{cases}
  \ds \mathcal{E}(u,v) = \lim_{t \downarrow 0} \frac{1}{t} (u - T_t u, v), \medskip\\
  \ds \c{F} = \left\{u \in L^2(E;m) : \lim_{t \downarrow 0} \frac{1}{t}
  (u - T_t u, u) < \infty\right\}.
 \end{cases}
\end{align}

We define the resolvent \(R_\alpha\) for \(\alpha>0\) by
\[R_{\alpha}f :=\int_0^\infty e^{-\alpha t} P_tf dt \]
for \(f\in L^2(E;m)\). We define the extended Dirichlet space \(\mathcal{F}_e\) as the set of all \(m\)-measurable functions \(f\) on \(E\) possessing an approximating \(\mathcal{E}\)-Cauchy sequence \(\{f_n\}_n \in \mathbb{N} \subset \mathcal{F}\) such that \(f_n\) converges to \(f\) \(m\)-almost everywhere.

Next we describe the capacity associated with $(\c{E},\c{F})$. For an open set $A \subset E$, the capacity of \(A\) is defined as
\[ \Capa(A) := \inf\big\{\c{E}_1(u,u) : u \in \mathcal{F}\ \text{with\ }u \ge 1\ m\text{-a.e. on } A\big\}, \]
where we adopt that the infimum of the empty set is \(\infty\).  For any set $B \subset E$, the capacity of \(B\) is defined as  
\[\Capa(B):= \inf\{\Capa(A) : B\subset A \ \text{for\ an\ open\ set\ }A\}.\]

We call $A \subset E$ is an {\it exceptional} set if $\Capa(A) = 0$.  A statement depending on $x \in A$ is said to hold  {\it quasi-everywhere} ({\it q.e.})  on $A$
if there exists an exceptional set $N \subset A$ such that the statement holds for every 
$x \in A \setminus N$, and a \([-\infty, \infty]\)-valued function $u$ defined q.e. on $E$ is {\it quasi-continuous} if, for any $\varepsilon > 0$, there exists an open set $U \subset E$ such that $\Capa(U) < \varepsilon$ and $u$ is finite and continuous on $E\setminus U$. By \cite[Theorem 2.1.3]{FOT11}, every function $u \in \c{F}_e$ has a quasi-continuous version. Therefore, in the rest of this paper, we always assume that any function $u \in  \c{F}_{e}$ is quasi-continuous.

We call an increasing sequence of closed sets $\{F_k\}$ a {\it nest} if 
\(\Capa (K\setminus F_k)\) converges to \(0\) as \(k\) tend to infinity for any compact set \(K\). We remark that this nest is called a generalized nest in \cite{FOT11}. See \cite[Section 2.2]{FOT11} and \cite[Section 2.3]{CF12} for details. A positive Borel measure $\mu$ on $E$ is {\it smooth} if it charges no set of zero capacity and there exists a nest \(\{F_k\}_ k \in \mathbb{N}\) satisfying \(\mu(F_k) <\infty\) for each \(k\). Denote by $\c{S}$ the class of all smooth measures.

Next we introduce the definition of a
positive continuous additive functional (PCAF).

\begin{definition} \rm
 \label{def-PCAF}
 A \([0,\infty]\)-valued function
 $A_t(\omega),\ t \ge 0,\ \omega \in \Omega$
 is a PCAF if the following hold:
 \begin{itemize}
  \item[(A.1)] for each $t \ge 0$, $A_t(\cdot)$ is $\c{F}_t$-measurable,
  \item[(A.2)]
       there exist a set $\Lambda \in \c{F}_{\infty}$ called a defining set, and
       an exceptional set $N \subset E$ such that $\b{P}_x(\Lambda) = 1$
       for any $x \in E \setminus N$ and $\theta_t \Lambda \subset \Lambda$
       for any $t > 0$, and moreover,
       for each $\omega \in \Lambda$, ${A}_0(\omega) = 0$,
       $|A_t(\omega)| < \infty$ for any $t < \zeta(\omega)$,
       $A_t(\omega) = {A}_{\zeta(\omega)}(\omega)$ for any
       $t \ge \zeta(\omega)$ and
       \begin{align}
	{A}_{t + s}(\omega) = {A}_s(\omega) + {A}_t(\theta_s \omega),\
	\text{for any $t,s \ge 0$},
       \end{align}
  \item[(A.3)]
       for each $\omega \in \Lambda$, the map
       $t \mapsto {A}_t(\omega)$ is non-negative and continuous on
       $[0,\infty)$, where $\Lambda$ is the set defined in the above condition.
 \end{itemize}
\end{definition}
The set of all PCAFs is denoted by $\bf{A}_c^+$. Two PCAFs $A^{(1)}$ and $A^{(2)}$ are {\it equivalent} if for each $t > 0$, $\b{P}_x(A_t^{(1)} = A_t^{(2)}) = 1$ q.e. $x \in E$, and we write $A^{(1)} \sim A^{(2)}$. This relation ``$\sim$'' is an equivalence relation on $\bf{A}_c^+$.
It is known that a PCAF corresponds one-to-one to a smooth measure as follows.
\begin{theorem}[{see \cite[Theorem 5.1.3, 5.1.4]{FOT11}}]
 \label{thm-Revuz}
 The quotient space $\bf{A}_c^+/\sim$ and the family $\c{S}$ are in one-to-one correspondence under the following relation:
 For $\mu \in \c{S}$ and $A \in \bf{A}_c^+$,
 \begin{align} \label{eq-Revuz}
  \mathbb{E}_{hm} \left[ \int_0^\infty e^{-\alpha s} f(X_s) \, dA_s \right] 
  = \int_E f(x) R_\alpha h(x)\, \mu(dx) 
 \end{align}
for any non-negative valued Borel measurable functions \(f\) and \(h\). Moreover, $(\ref{eq-Revuz})$ is equivalent to 
\begin{equation}
 \mathbb{E}_{hm} \left[ \int_0^t  f(X_s) \, dA_s \right] =
   \int_{E} \left(\int_{0}^{t} P_s h(x)\, ds\right) f(x)\, \mu(dx)
\label{eq-Revuz2}
\end{equation}
 for any non-negative valued Borel measurable functions \(f\) and \(h\).
 \end{theorem}

The equation \eqref{eq-Revuz}  or \eqref{eq-Revuz2} is called the {\it Revuz correspondence}.

\ \\
{\bf Acknowledgement.} This work was supported by JSPS KAKENHI Grant Numbers  25K17270 (T.O.) and 25K07056 (T.U.).

\end{document}